%% file: main.tex
\documentclass{article}
\input{Appendix/smallCasesCompletePreamble}

\usepackage[english]{babel}

\usepackage{amsmath}
\usepackage{amsthm}
\usepackage{graphicx}
\usepackage[colorlinks=true, allcolors=blue]{hyperref}
\usepackage{amsfonts}
\usepackage{cleveref}
\usepackage{ amssymb }
\usepackage[shortlabels]{enumitem}
\usepackage{tikz}
\usetikzlibrary{positioning,decorations.pathmorphing,shapes}
\tikzset{main node/.style={circle,draw,minimum size=.6cm,inner sep=0pt},}
\usepackage{comment}
\usepackage{subcaption}
\usepackage{algorithm}
\usepackage{algpseudocode}
\usetikzlibrary{positioning}
\usetikzlibrary{calc}
\usetikzlibrary{arrows}
\usetikzlibrary{decorations.markings}
\usepackage{thm-restate}

\newcommand*{\StrikeThruDistance}{0.15cm}%

\newcommand{\cut}{\mathrm{cat}}
\newcommand{\cat}{\cut}
\newtheorem{theorem}{Theorem}[section]
\newtheorem{corollary}[theorem]{Corollary}
\newtheorem{lemma}[theorem]{Lemma}
\newtheorem{definition}[theorem]{Definition}

\crefname{lemma}{Lemma}{Lemmas}
\counterwithin{figure}{section}

\title{Cuts, Cats, and Complete Graphs}
\author{Rylo Ashmore, Danny Dyer, Trent Marbach, Rebecca Milley}

\begin{document}
\maketitle

\input{Sections/abstract}
\input{Sections/intro}
\input{Sections/definitions}
\input{Sections/initialResultsAndCentrality}
\input{Sections/families}
\input{Sections/complete}
\input{Sections/openQs}

\bibliographystyle{abbrv}
\bibliography{main}
\pagebreak
\appendix
\section{Small Cases Code}
\input{Appendix/appendixDiscuss}

\input{Appendix/smallCasesComplete}
\end{document}

%% file: Appendix/smallCasesCompletePreamble.tex
    \usepackage[breakable]{tcolorbox}
    \usepackage{parskip} 

    \usepackage{graphicx}
    
    \usepackage{caption}

    \usepackage{float}
    \usepackage{xcolor} 
    \usepackage{enumerate} 
    \usepackage[letterpaper,top=2cm,bottom=2cm,left=3cm,right=3cm,marginparwidth=1.75cm]{geometry} 
    \usepackage{amsmath} 
    \usepackage{amssymb} 
    \usepackage{textcomp} 
    \AtBeginDocument{%
        \def\PYZsq{\textquotesingle}
    }
    \usepackage{upquote} 
    \usepackage{eurosym} 

    \usepackage{iftex}
    \ifPDFTeX
        \usepackage[T1]{fontenc}
        \IfFileExists{alphabeta.sty}{
              \usepackage{alphabeta}
          }{
              \usepackage[mathletters]{ucs}
              \usepackage[utf8x]{inputenc}
          }
    \else
        \usepackage{fontspec}
        \usepackage{unicode-math}
    \fi

    \usepackage{fancyvrb} 
    \usepackage{grffile} 
    \makeatletter 
    \@ifpackagelater{grffile}{2019/11/01}
    {
    }
    {
      \def\Gread@@xetex#1{%
        \IfFileExists{"\Gin@base".bb}%
        {\Gread@eps{\Gin@base.bb}}%
        {\Gread@@xetex@aux#1}%
      }
    }
    \makeatother
    \usepackage[Export]{adjustbox} 
    \adjustboxset{max size={0.9\linewidth}{0.9\paperheight}}

    \usepackage[colorlinks=true, linkcolor=blue,allcolors=blue]{hyperref}
    \usepackage{titling}
    \usepackage{longtable} 
    \usepackage{booktabs}  
    \usepackage{array}     
    \usepackage{calc}      
    \usepackage[inline,shortlabels]{enumitem} 
    \usepackage[normalem]{ulem} 
    \usepackage{mathrsfs}

    \definecolor{urlcolor}{rgb}{0,.145,.698}
    \definecolor{linkcolor}{rgb}{.71,0.21,0.01}
    \definecolor{citecolor}{rgb}{.12,.54,.11}

    \definecolor{ansi-black}{HTML}{3E424D}
    \definecolor{ansi-black-intense}{HTML}{282C36}
    \definecolor{ansi-red}{HTML}{E75C58}
    \definecolor{ansi-red-intense}{HTML}{B22B31}
    \definecolor{ansi-green}{HTML}{00A250}
    \definecolor{ansi-green-intense}{HTML}{007427}
    \definecolor{ansi-yellow}{HTML}{DDB62B}
    \definecolor{ansi-yellow-intense}{HTML}{B27D12}
    \definecolor{ansi-blue}{HTML}{208FFB}
    \definecolor{ansi-blue-intense}{HTML}{0065CA}
    \definecolor{ansi-magenta}{HTML}{D160C4}
    \definecolor{ansi-magenta-intense}{HTML}{A03196}
    \definecolor{ansi-cyan}{HTML}{60C6C8}
    \definecolor{ansi-cyan-intense}{HTML}{258F8F}
    \definecolor{ansi-white}{HTML}{C5C1B4}
    \definecolor{ansi-white-intense}{HTML}{A1A6B2}
    \definecolor{ansi-default-inverse-fg}{HTML}{FFFFFF}
    \definecolor{ansi-default-inverse-bg}{HTML}{000000}

    \definecolor{outerrorbackground}{HTML}{FFDFDF}

    \DefineVerbatimEnvironment{Highlighting}{Verbatim}{commandchars=\\\{\}}

    \let\Oldtex\TeX
    \let\Oldlatex\LaTeX
    \renewcommand{\TeX}{\textrm{\Oldtex}}
    \renewcommand{\LaTeX}{\textrm{\Oldlatex}}
    \title{smallCasesComplete}

\makeatletter
\def\PY@reset{\let\PY@it=\relax \let\PY@bf=\relax%
    \let\PY@ul=\relax \let\PY@tc=\relax%
    \let\PY@bc=\relax \let\PY@ff=\relax}
\def\PY@tok#1{\csname PY@tok@#1\endcsname}
\def\PY@toks#1+{\ifx\relax#1\empty\else%
    \PY@tok{#1}\expandafter\PY@toks\fi}
\def\PY@do#1{\PY@bc{\PY@tc{\PY@ul{%
    \PY@it{\PY@bf{\PY@ff{#1}}}}}}}
\def\PY#1#2{\PY@reset\PY@toks#1+\relax+\PY@do{#2}}

\@namedef{PY@tok@w}{\def\PY@tc##1{\textcolor[rgb]{0.73,0.73,0.73}{##1}}}
\@namedef{PY@tok@c}{\let\PY@it=\textit\def\PY@tc##1{\textcolor[rgb]{0.24,0.48,0.48}{##1}}}
\@namedef{PY@tok@cp}{\def\PY@tc##1{\textcolor[rgb]{0.61,0.40,0.00}{##1}}}
\@namedef{PY@tok@k}{\let\PY@bf=\textbf\def\PY@tc##1{\textcolor[rgb]{0.00,0.50,0.00}{##1}}}
\@namedef{PY@tok@kp}{\def\PY@tc##1{\textcolor[rgb]{0.00,0.50,0.00}{##1}}}
\@namedef{PY@tok@kt}{\def\PY@tc##1{\textcolor[rgb]{0.69,0.00,0.25}{##1}}}
\@namedef{PY@tok@o}{\def\PY@tc##1{\textcolor[rgb]{0.40,0.40,0.40}{##1}}}
\@namedef{PY@tok@ow}{\let\PY@bf=\textbf\def\PY@tc##1{\textcolor[rgb]{0.67,0.13,1.00}{##1}}}
\@namedef{PY@tok@nb}{\def\PY@tc##1{\textcolor[rgb]{0.00,0.50,0.00}{##1}}}
\@namedef{PY@tok@nf}{\def\PY@tc##1{\textcolor[rgb]{0.00,0.00,1.00}{##1}}}
\@namedef{PY@tok@nc}{\let\PY@bf=\textbf\def\PY@tc##1{\textcolor[rgb]{0.00,0.00,1.00}{##1}}}
\@namedef{PY@tok@nn}{\let\PY@bf=\textbf\def\PY@tc##1{\textcolor[rgb]{0.00,0.00,1.00}{##1}}}
\@namedef{PY@tok@ne}{\let\PY@bf=\textbf\def\PY@tc##1{\textcolor[rgb]{0.80,0.25,0.22}{##1}}}
\@namedef{PY@tok@nv}{\def\PY@tc##1{\textcolor[rgb]{0.10,0.09,0.49}{##1}}}
\@namedef{PY@tok@no}{\def\PY@tc##1{\textcolor[rgb]{0.53,0.00,0.00}{##1}}}
\@namedef{PY@tok@nl}{\def\PY@tc##1{\textcolor[rgb]{0.46,0.46,0.00}{##1}}}
\@namedef{PY@tok@ni}{\let\PY@bf=\textbf\def\PY@tc##1{\textcolor[rgb]{0.44,0.44,0.44}{##1}}}
\@namedef{PY@tok@na}{\def\PY@tc##1{\textcolor[rgb]{0.41,0.47,0.13}{##1}}}
\@namedef{PY@tok@nt}{\let\PY@bf=\textbf\def\PY@tc##1{\textcolor[rgb]{0.00,0.50,0.00}{##1}}}
\@namedef{PY@tok@nd}{\def\PY@tc##1{\textcolor[rgb]{0.67,0.13,1.00}{##1}}}
\@namedef{PY@tok@s}{\def\PY@tc##1{\textcolor[rgb]{0.73,0.13,0.13}{##1}}}
\@namedef{PY@tok@sd}{\let\PY@it=\textit\def\PY@tc##1{\textcolor[rgb]{0.73,0.13,0.13}{##1}}}
\@namedef{PY@tok@si}{\let\PY@bf=\textbf\def\PY@tc##1{\textcolor[rgb]{0.64,0.35,0.47}{##1}}}
\@namedef{PY@tok@se}{\let\PY@bf=\textbf\def\PY@tc##1{\textcolor[rgb]{0.67,0.36,0.12}{##1}}}
\@namedef{PY@tok@sr}{\def\PY@tc##1{\textcolor[rgb]{0.64,0.35,0.47}{##1}}}
\@namedef{PY@tok@ss}{\def\PY@tc##1{\textcolor[rgb]{0.10,0.09,0.49}{##1}}}
\@namedef{PY@tok@sx}{\def\PY@tc##1{\textcolor[rgb]{0.00,0.50,0.00}{##1}}}
\@namedef{PY@tok@m}{\def\PY@tc##1{\textcolor[rgb]{0.40,0.40,0.40}{##1}}}
\@namedef{PY@tok@gh}{\let\PY@bf=\textbf\def\PY@tc##1{\textcolor[rgb]{0.00,0.00,0.50}{##1}}}
\@namedef{PY@tok@gu}{\let\PY@bf=\textbf\def\PY@tc##1{\textcolor[rgb]{0.50,0.00,0.50}{##1}}}
\@namedef{PY@tok@gd}{\def\PY@tc##1{\textcolor[rgb]{0.63,0.00,0.00}{##1}}}
\@namedef{PY@tok@gi}{\def\PY@tc##1{\textcolor[rgb]{0.00,0.52,0.00}{##1}}}
\@namedef{PY@tok@gr}{\def\PY@tc##1{\textcolor[rgb]{0.89,0.00,0.00}{##1}}}
\@namedef{PY@tok@ge}{\let\PY@it=\textit}
\@namedef{PY@tok@gs}{\let\PY@bf=\textbf}
\@namedef{PY@tok@gp}{\let\PY@bf=\textbf\def\PY@tc##1{\textcolor[rgb]{0.00,0.00,0.50}{##1}}}
\@namedef{PY@tok@go}{\def\PY@tc##1{\textcolor[rgb]{0.44,0.44,0.44}{##1}}}
\@namedef{PY@tok@gt}{\def\PY@tc##1{\textcolor[rgb]{0.00,0.27,0.87}{##1}}}
\@namedef{PY@tok@err}{\def\PY@bc##1{{\setlength{\fboxsep}{\string -\fboxrule}\fcolorbox[rgb]{1.00,0.00,0.00}{1,1,1}{\strut ##1}}}}
\@namedef{PY@tok@kc}{\let\PY@bf=\textbf\def\PY@tc##1{\textcolor[rgb]{0.00,0.50,0.00}{##1}}}
\@namedef{PY@tok@kd}{\let\PY@bf=\textbf\def\PY@tc##1{\textcolor[rgb]{0.00,0.50,0.00}{##1}}}
\@namedef{PY@tok@kn}{\let\PY@bf=\textbf\def\PY@tc##1{\textcolor[rgb]{0.00,0.50,0.00}{##1}}}
\@namedef{PY@tok@kr}{\let\PY@bf=\textbf\def\PY@tc##1{\textcolor[rgb]{0.00,0.50,0.00}{##1}}}
\@namedef{PY@tok@bp}{\def\PY@tc##1{\textcolor[rgb]{0.00,0.50,0.00}{##1}}}
\@namedef{PY@tok@fm}{\def\PY@tc##1{\textcolor[rgb]{0.00,0.00,1.00}{##1}}}
\@namedef{PY@tok@vc}{\def\PY@tc##1{\textcolor[rgb]{0.10,0.09,0.49}{##1}}}
\@namedef{PY@tok@vg}{\def\PY@tc##1{\textcolor[rgb]{0.10,0.09,0.49}{##1}}}
\@namedef{PY@tok@vi}{\def\PY@tc##1{\textcolor[rgb]{0.10,0.09,0.49}{##1}}}
\@namedef{PY@tok@vm}{\def\PY@tc##1{\textcolor[rgb]{0.10,0.09,0.49}{##1}}}
\@namedef{PY@tok@sa}{\def\PY@tc##1{\textcolor[rgb]{0.73,0.13,0.13}{##1}}}
\@namedef{PY@tok@sb}{\def\PY@tc##1{\textcolor[rgb]{0.73,0.13,0.13}{##1}}}
\@namedef{PY@tok@sc}{\def\PY@tc##1{\textcolor[rgb]{0.73,0.13,0.13}{##1}}}
\@namedef{PY@tok@dl}{\def\PY@tc##1{\textcolor[rgb]{0.73,0.13,0.13}{##1}}}
\@namedef{PY@tok@s2}{\def\PY@tc##1{\textcolor[rgb]{0.73,0.13,0.13}{##1}}}
\@namedef{PY@tok@sh}{\def\PY@tc##1{\textcolor[rgb]{0.73,0.13,0.13}{##1}}}
\@namedef{PY@tok@s1}{\def\PY@tc##1{\textcolor[rgb]{0.73,0.13,0.13}{##1}}}
\@namedef{PY@tok@mb}{\def\PY@tc##1{\textcolor[rgb]{0.40,0.40,0.40}{##1}}}
\@namedef{PY@tok@mf}{\def\PY@tc##1{\textcolor[rgb]{0.40,0.40,0.40}{##1}}}
\@namedef{PY@tok@mh}{\def\PY@tc##1{\textcolor[rgb]{0.40,0.40,0.40}{##1}}}
\@namedef{PY@tok@mi}{\def\PY@tc##1{\textcolor[rgb]{0.40,0.40,0.40}{##1}}}
\@namedef{PY@tok@il}{\def\PY@tc##1{\textcolor[rgb]{0.40,0.40,0.40}{##1}}}
\@namedef{PY@tok@mo}{\def\PY@tc##1{\textcolor[rgb]{0.40,0.40,0.40}{##1}}}
\@namedef{PY@tok@ch}{\let\PY@it=\textit\def\PY@tc##1{\textcolor[rgb]{0.24,0.48,0.48}{##1}}}
\@namedef{PY@tok@cm}{\let\PY@it=\textit\def\PY@tc##1{\textcolor[rgb]{0.24,0.48,0.48}{##1}}}
\@namedef{PY@tok@cpf}{\let\PY@it=\textit\def\PY@tc##1{\textcolor[rgb]{0.24,0.48,0.48}{##1}}}
\@namedef{PY@tok@c1}{\let\PY@it=\textit\def\PY@tc##1{\textcolor[rgb]{0.24,0.48,0.48}{##1}}}
\@namedef{PY@tok@cs}{\let\PY@it=\textit\def\PY@tc##1{\textcolor[rgb]{0.24,0.48,0.48}{##1}}}

\def\PYZbs{\char`\\}
\def\PYZus{\char`\_}
\def\PYZob{\char`\{}
\def\PYZcb{\char`\}}

\def\PYZlt{\char`\<}
\def\PYZgt{\char`\>}
\def\PYZsh{\char`\#}
\def\PYZpc{\char`\%}

\def\PYZhy{\char`\-}
\def\PYZsq{\char`\'}
\def\PYZdq{\char`\"}

\makeatother

    \makeatletter
        \newbox\Wrappedcontinuationbox
        \newbox\Wrappedvisiblespacebox
        \newcommand*\Wrappedvisiblespace {\textcolor{red}{\textvisiblespace}}
        \newcommand*\Wrappedcontinuationsymbol {\textcolor{red}{\llap{\tiny$\m@th\hookrightarrow$}}}
        \newcommand*\Wrappedcontinuationindent {3ex }
        \newcommand*\Wrappedafterbreak {\kern\Wrappedcontinuationindent\copy\Wrappedcontinuationbox}
        \newcommand*\Wrappedbreaksatspecials {%
            \def\PYGZus{\discretionary{\char`\_}{\Wrappedafterbreak}{\char`\_}}%
            \def\PYGZob{\discretionary{}{\Wrappedafterbreak\char`\{}{\char`\{}}%
            \def\PYGZcb{\discretionary{\char`\}}{\Wrappedafterbreak}{\char`\}}}%
            \def\PYGZca{\discretionary{\char`\^}{\Wrappedafterbreak}{\char`\^}}%
            \def\PYGZam{\discretionary{\char`\&}{\Wrappedafterbreak}{\char`\&}}%
            \def\PYGZlt{\discretionary{}{\Wrappedafterbreak\char`\<}{\char`\<}}%
            \def\PYGZgt{\discretionary{\char`\>}{\Wrappedafterbreak}{\char`\>}}%
            \def\PYGZsh{\discretionary{}{\Wrappedafterbreak\char`\#}{\char`\#}}%
            \def\PYGZpc{\discretionary{}{\Wrappedafterbreak\char`\%}{\char`\%}}%
            \def\PYGZdl{\discretionary{}{\Wrappedafterbreak\char`\$}{\char`\$}}%
            \def\PYGZhy{\discretionary{\char`\-}{\Wrappedafterbreak}{\char`\-}}%
            \def\PYGZsq{\discretionary{}{\Wrappedafterbreak\textquotesingle}{\textquotesingle}}%
            \def\PYGZdq{\discretionary{}{\Wrappedafterbreak\char`\"}{\char`\"}}%
            \def\PYGZti{\discretionary{\char`\~}{\Wrappedafterbreak}{\char`\~}}%
        }
        \newcommand*\Wrappedbreaksatpunct {%
            \lccode`\~`\.\lowercase{\def~}{\discretionary{\hbox{\char`\.}}{\Wrappedafterbreak}{\hbox{\char`\.}}}%
            \lccode`\~`\,\lowercase{\def~}{\discretionary{\hbox{\char`\,}}{\Wrappedafterbreak}{\hbox{\char`\,}}}%
            \lccode`\~`\;\lowercase{\def~}{\discretionary{\hbox{\char`\;}}{\Wrappedafterbreak}{\hbox{\char`\;}}}%
            \lccode`\~`\:\lowercase{\def~}{\discretionary{\hbox{\char`\:}}{\Wrappedafterbreak}{\hbox{\char`\:}}}%
            \lccode`\~`\?\lowercase{\def~}{\discretionary{\hbox{\char`\?}}{\Wrappedafterbreak}{\hbox{\char`\?}}}%
            \lccode`\~`\!\lowercase{\def~}{\discretionary{\hbox{\char`\!}}{\Wrappedafterbreak}{\hbox{\char`\!}}}%
            \lccode`\~`\/\lowercase{\def~}{\discretionary{\hbox{\char`\/}}{\Wrappedafterbreak}{\hbox{\char`\/}}}%
            \catcode`\.\active
            \catcode`\,\active
            \catcode`\;\active
            \catcode`\:\active
            \catcode`\?\active
            \catcode`\!\active
            \catcode`\/\active
            \lccode`\~`\~
        }
    \makeatother

    \let\OriginalVerbatim=\Verbatim
    \makeatletter
    \renewcommand{\Verbatim}[1][1]{%
        \sbox\Wrappedcontinuationbox {\Wrappedcontinuationsymbol}%
        \sbox\Wrappedvisiblespacebox {\FV@SetupFont\Wrappedvisiblespace}%
        \def\FancyVerbFormatLine ##1{\hsize\linewidth
            \vtop{\raggedright\hyphenpenalty\z@\exhyphenpenalty\z@
                \doublehyphendemerits\z@\finalhyphendemerits\z@
                \strut ##1\strut}%
        }%
        \def\FV@Space {%
            \nobreak\hskip\z@ plus\fontdimen3\font minus\fontdimen4\font
            \discretionary{\copy\Wrappedvisiblespacebox}{\Wrappedafterbreak}
            {\kern\fontdimen2\font}%
        }%

        \Wrappedbreaksatspecials
        \OriginalVerbatim[#1,codes*=\Wrappedbreaksatpunct]%
    }
    \makeatother

    \definecolor{incolor}{HTML}{303F9F}
    \definecolor{outcolor}{HTML}{D84315}
    \definecolor{cellborder}{HTML}{CFCFCF}
    \definecolor{cellbackground}{HTML}{F7F7F7}

    \makeatletter
    \newcommand{\boxspacing}{\kern\kvtcb@left@rule\kern\kvtcb@boxsep}
    \makeatother
    \newcommand{\prompt}[4]{
        {\ttfamily\llap{{\color{#2}[#3]:\hspace{3pt}#4}}\vspace{-\baselineskip}}
    }

    \hypersetup{
      breaklinks=true,  
      colorlinks=true,
      urlcolor=urlcolor,
      }
    

%% file: Sections/abstract.tex
\begin{abstract}
We introduce the game of Cat Herding, where an omnipresent herder slowly cuts down a graph until an evasive cat player has nowhere to go. The number of cuts made is the score of a game, and we study the score under optimal play. In this paper, we begin by deriving some general results, and then we determine the precise cat number for paths, cycles, stars, and wheels. Finally, we identify an optimal Cat and Herder strategy on complete graphs, while providing both a recurrence and closed form for $\cat(K_n)$. 
\end{abstract}

%% file: Sections/intro.tex
\section{Introduction}
The published board game \emph{Nowhere To Go}\cite{Atkins_2012} is a combinatorial game in which each player is on a vertex of a network. On each player's turn, they move to a new vertex along any path not through the other player, and then delete any edge. Play proceeds until a player cannot move along a path, and they have ``Nowhere To Go''. A notable endgame in this game is when each player ends up in their own component. In this case, each component has one player deleting edges and the other moving along non-trivial paths. As play continues, whichever component permits the player to evade capture the longest will win. We introduce the game of Cat Herding, which can model this endgame. One player (the cat) will move their cat token along non-trivial paths and the other (the herder) will delete edges. The score of a game is the total number of edge deletions that are necessary to isolate the cat. The herder aims to minimize the score, while the cat aims to maximize it, and the score when both play optimally is the \emph{cat number} of the graph.

As this paper introduces a novel pursuit-evasion game, we examine some similarities and differences to pre-existing pursuit-evasion games. A classic pursuit-evasion game on graphs is Cops and Robbers. In that game, both players are present on vertices of the graph, and a graph is $k$-cop-win when $k$ cops have a strategy to guarantee some cop eventually occupies the same space as the robber. This is similar to Cat Herding, as the cat is on a vertex of the graph. The graphs where the one cop is sufficient to capture the robber (`copwin' graphs) were first classified structurally in \cite{Nowakowski1983VertextovertexPI}, and then later described algorithmically in \cite{AIGNER19841}, and analyzed by genus in \cite{QUILLIOT198589}. Cat herding differs from Cops and Robbers as the pursuer (the herder) plays omnipresently. Omnipresent play also exists in pre-existing literature, such as in the Angel and Devil game presented in \cite{berlekamp_2018}, where the Devil omnipresently chooses a vertex to delete, and the Angel can move up to distance $k$. The Devil wins if they can eventually capture the Angel, and the Angel wins if they can evade capture forever. The Angel and Devil game has sparked multiple papers seeking the barrier between when the Angel and Devil win in these games\cite{article,10.1017/S0963548306008303,KlosterAngelSolution}. The Angel and Devil problem has many variants, including increasing dimension \cite{MR2192775}, or changing the Angel to walk rather than jump \cite{MR2190915}. Comparing to the game of Cat Herding, there is a notable difference in that the herder deletes edges, while the Devil deletes vertices. Finally, we note that the cat has infinite velocity, as they may move along any non-trivial path on their turn. A variant of Cops and Robbers with an infinite velocity (visible) robber is analyzed in \cite{SEYMOUR199322}, although in this game the pursuers capture the evader by moving onto them as opposed to isolating them through structural deletions. We also note that the localization game introduced in \cite{SEAGER20123265}, in which distance probes are placed in the graph every turn, may be viewed as the localizer playing omnipresently, or the probes moving at infinite velocity.

%% file: Sections/definitions.tex
\section{Definitions}

For standard graph definitions of paths, cycles, and complete graphs, we defer to \cite{west_introduction_2000}.

The game of Cat Herding is played on a graph $G$ between two players, the \emph{cat} and the 
\emph{herder}. The game's setup consists of the cat choosing a starting vertex for their cat token, which we will also refer to as the cat. After this, both players will alternate turns, beginning with the herder: they will delete (any) one edge, called a \emph{cut}, and the cat moves along a path to a new vertex. Note that in this context, a cut refers to a single edge cut, which may or may not disconnect the graph. Note also that the cat must move on its turn. Once the cat is on an isolated vertex and cannot move (which must happen within $|E(G)|$ cuts), the game is done. The objective of the herder is to minimize the number of cuts required to isolate the cat, while the objective of the cat is to maximize the number of cuts required to isolate them. When both players play such that no deviation can improve the score (with respect to their objective), we say that the game is optimally played. We will also refer to the cat having no moves available as being \emph{captured}.

\begin{definition}
    For $v\in V(G)$, we define the cat number of a vertex $\cat(G,v)$ to be the number of cuts required to isolate the cat in an optimally played game in $G$ where the cat starts at $v$. We define the cat number of a graph $\cat(G)=\max_{v\in V(G)}\cat(G,v)$.
\end{definition}

\newcounter{j}
\tikzset{strike thru arrow/.style={
    decoration={markings, mark=at position 0.5 with {
        \draw [thick,-] 
            ++ (-\StrikeThruDistance,-\StrikeThruDistance) 
            -- ( \StrikeThruDistance, \StrikeThruDistance);}
    },
    postaction={decorate},
}}
\newcommand{\subfigSpace}[0]{\vspace{.3cm}}
\newcommand{\eDist}{.15cm}
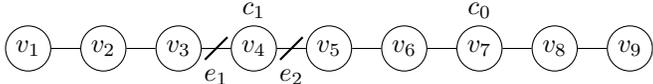
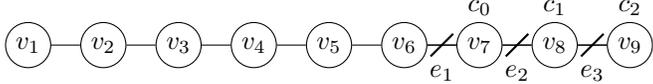
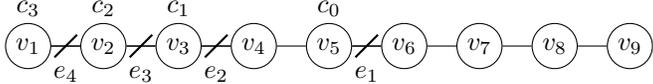
\begin{figure}
     \centering
     \begin{subfigure}[c]{\textwidth}
         \centering
            \begin{tikzpicture}
                \foreach \i in {1,...,9}{
                    \node[main node] (\i) at (\i,0) {$v_\i$};
                }
                \draw (1) -- (2);
                \draw (2) -- (3);
                \draw (3) -- (4);
                \draw (4) -- (5);
                \draw (5) -- (6);
                \draw (6) -- (7);
                \draw (7) -- (8);
                \draw (8) -- (9);
                
                \node at (7,.5) {$c_0$};
                \draw[strike thru arrow] (3) -- (4) node [midway, below=\eDist] {$e_1$};
                \node at (4,.5) {$c_1$};
                \draw[strike thru arrow] (4) -- (5) node [midway, below=\eDist] {$e_2$};
            \end{tikzpicture}
         \caption{Sub-optimal play by both players, with cat starting on $v_7$.}
         \label{fig:subOptEx}
     \end{subfigure}
     \begin{subfigure}[c]{\textwidth}
        \subfigSpace
        \centering
         \begin{tikzpicture}
                \foreach \i in {1,...,9}{
                    \node[main node] (\i) at (\i,0) {$v_\i$};
                }
                \draw (1) -- (2);
                \draw (2) -- (3);
                \draw (3) -- (4);
                \draw (4) -- (5);
                \draw (5) -- (6);
                \draw (6) -- (7);
                \draw (7) -- (8);
                \draw (8) -- (9);
                
                \node at (7,.5) {$c_0$};
                \draw[strike thru arrow] (6) -- (7) node [midway, below=\eDist] {$e_1$};
                \node at (8,.5) {$c_1$};
                \draw[strike thru arrow] (7) -- (8) node [midway, below=\eDist] {$e_2$};
                \node at (9,.5) {$c_2$};
                \draw[strike thru arrow] (9) -- (8) node [midway, below=\eDist] {$e_3$};
            \end{tikzpicture}
         \caption{Optimal play when cat starts on $v_7$}
         \label{fig:optV7}
     \end{subfigure}
     \begin{subfigure}[c]{\textwidth}
        \subfigSpace
        \centering
         \begin{tikzpicture}
                \foreach \i in {1,...,9}{
                    \node[main node] (\i) at (\i,0) {$v_\i$};
                }
                \draw (1) -- (2);
                \draw (2) -- (3);
                \draw (3) -- (4);
                \draw (4) -- (5);
                \draw (5) -- (6);
                \draw (6) -- (7);
                \draw (7) -- (8);
                \draw (8) -- (9);
                
                \node at (5,.5) {$c_0$};
                \draw[strike thru arrow] (5) -- (6) node [midway, below=\eDist] {$e_1$};
                \node at (3,.5) {$c_1$};
                \draw[strike thru arrow] (3) -- (4) node [midway, below=\eDist] {$e_2$};
                \node at (2,.5) {$c_2$};
                \draw[strike thru arrow] (2) -- (3) node [midway, below=\eDist] {$e_3$};
                \node at (1,.5) {$c_3$};
                \draw[strike thru arrow] (1) -- (2) node [midway, below=\eDist] {$e_4$};        
            \end{tikzpicture}
         \caption{Optimal play with optimal starting vertex $v_5$}
         \label{fig:optV5}
     \end{subfigure}
    \caption{Three possible games on $P_9$. Vertex $c_i$ denotes the cat's position on turn $i$, and $e_i$ denotes the $i$th edge that the herder cut.}
    \label{fig:three games}
\end{figure}

For example, consider play on the graph $P_9$, as shown in \Cref{fig:three games}. In \Cref{fig:subOptEx}, which illustrates suboptimal play by both players, the cat makes a sub-optimal opening move $v_7$. The herder responds with a sub-optimal move by cutting edge $v_3v_4$. The cat then makes a sub-optimal move to the leaf $v_4$, and the herder now makes an optimal move of deleting $v_4v_5$, isolating the cat in $2$ cuts. Because neither player played optimally, we do not conclude anything about $\cat(P_9)$.

In \Cref{fig:optV7} diagramming optimal play where the cat is forced to start at $v_7$, the herder first cuts $v_6v_7$. The cat moves to $v_8$, and the herder cuts $v_7v_8$. The cat moves to $v_9$ as the only move, and the herder cuts $v_8v_9$, isolating the cat in $3$ cuts. We will later see that this is an optimal game when the cat starts at $v_7$, so that $\cat(P_9,v_7)=3$.

In \Cref{fig:optV5} diagramming optimal play, the cat starts by choosing to start on $v_5$. The herder deletes $v_5v_6$, and the cat moves to the center of the resulting component, $v_3$. The herder cuts $v_3v_4$, and the cat moves to the center of the resulting component, $v_2$. The herder cuts $v_2v_3$, and the cat makes their only move to $v_1$. After this, the herder cuts $v_1v_2$, isolating the cat in $4$ cuts. We will later see that this is an optimal game when the cat starts at $v_5$, so $\cat(P_9,v_5)=4$. Further, this is maximal across all vertices, so that $\cat(P_9)=4$.    

%% file: Sections/initialResultsAndCentrality.tex
\section{Initial Results}
We begin by considering the game generally, seeking some easy upper and lower bounds for the $\cat(G)$ parameter. In seeking upper and lower bounds, we note that we need to find cat/herder strategies that guarantee a good performance for their respective objectives.

\begin{lemma}\label{stratRules}
    If there exists a cat strategy $C$ guaranteeing survival for $k$ cuts, then $\cat(G)\geq k$. Similarly, if there exists a herder strategy $H$ guaranteeing capture in at most $k$ cuts, then $\cat(G)\leq k$.
\end{lemma}
\begin{proof}
    If $C$ is a cat strategy guaranteeing $k$ cut survival, then at worst the cat will choose strategy $C$ and attain $\cat(G)\geq k$ cuts. Better strategies may score better.
    A similar argument holds for herder strategies.
\end{proof}
Recall that a $k$-core of $G$ is a maximal subgraph $H$ with all vertices $v\in V(H)$ having $\deg(v)\geq k$.
\begin{lemma}\label{dumbLower}
    If $G$ is a connected graph with a $k$-core, then $\cat(G)\geq k$. 
\end{lemma}
\begin{proof}
    Let $H$ be a $k$-core of $G$. The cat's strategy is to voluntarily restrict play to $H$. If the cat is isolated on $v\in V(H)$, then all edges incident to $v$ must have been cut. There are at least $k$ such edges, giving the inequality.
\end{proof}
In particular, this shows that the cat number is unbounded when considered over all graphs. 

Once a graph is disconnected, the herder has a choice. They could play in the component the cat is in, or they may play outside of it. If the herder plays outside of the component the cat is in, we view this as a \emph{passing move}. We note that herder passing moves are never better than cutting in the cat's component. Namely, we show the following.
\begin{lemma}\label{noPass}
    If $G$ is a connected graph, $v\in V(G)$, and $e\in E(G)$, then $\cat(G,v)\geq \cat(G-e,v)$.
\end{lemma}

\begin{proof}
    We must provide a cat strategy in $G$ that will score at least $\cat(G-e,v)$. The cat's strategy in general is to mimic play in $G-e$, starting by choosing the same start vertex in $G$ as in $G-e$. On any herder move, if the herder cuts $e'\in G-e$, then the cat may query optimal strategy in $G-e$ to proceed. The only exception is if the herder chooses to cut $e'\in G\setminus (G-e)$, namely $e=e'$.

    At the moment the herder cuts $e'=e$, let $H$ be the subgraph the cat is left playing on. Let the cat be on $v\in V(H)$ with $\cat(H,v)$ as the remaining optimal score, and consider what the cat would do if the herder cut edge $e''\neq e$. If cutting edge $e''$ isolates the cat, then $\cat(H,v)=1$ and $\cat(H-e'',v)=0$, and adding the number of moves taken before $e'=e$ is cut to both sides, we find the result holds. Otherwise, the cat has some follow-up move that must score at least $\cat(H,v)-1$ in $H-e''$. Then it follows that there must be a path in $H-e''$ to a vertex $x$ such that $\cat(H-e'',x)\geq \cat(H,v)-1$ for the cat's move. If the herder cuts $e\not\in G-e$ as a `passing move' in $G$, then the cat may henceforth restrict themselves to $H-e''$ (by tying one hand behind their back) and attain $\cat(H-e'',x)\geq \cat(H,v)-1$ moves by moving to $x$. This is a cat strategy showing a herder's passing move is no better than cutting any edge $e$ in the graph, and so by \cref{stratRules}, $\cat(G,v)\geq \cat(G-e,v)$. 
\end{proof}
The distribution of $\cat(G,v)$ across all $v$ cannot contain a single vertex with a score significantly higher than all others. Namely, we show the following.
\begin{lemma}\label{noPass2}
    For all vertices $v\in V(G)$ in a connected graph $G\ncong K_1$, there exists a vertex $x$ with non-trivial path $P_{x,v}$ from $x$ to $v$ such that $\cat(G,x)\geq \cat(G,v)-1$. 
\end{lemma}
\begin{proof}
    Let $v\in V(G)$ and $e\in E(G)$. If $G-e$ isolates the cat then, $\cat(G,v)=1$, and $\cat(G,x)\geq 0$ gives us the result. Otherwise, in $G-e$ with the cat on $v$, the cat has some follow-up move, say to $x$. Then the cat moves to vertex $x$ along non-trivial $x-v$ path in $G-e$, scoring $\cat(G-e,x)$. Since $e$ was some herder move (which may or may not be optimal), we know that $\cat(G,v)\leq \cat(G-e,x)+1\leq \cat(G,x)+1$ by using \cref{noPass} for the final inequality.
\end{proof}
How do $\cat(G,v)$ and $1+\cat(G-e,v)$ relate? Namely, when would the cat trade an edge cut $e$ for a point added to the score, without moving? In \cref{noPass}, we showed that $\cat(G,v)\geq \cat(G-e,v)$. In fact, this gap can be arbitrarily large. Consider a cat on vertex $v$ on a $4$-cycle, where the vertex opposite $v$ is a member of a large clique, as in \Cref{fig:stallArbitrarilyBad}. If the cat traded a point added to the score for a cut, the herder would get two consecutive cuts and isolate the cat. Meanwhile, the alternative would be that the cat could escape to the large clique and obtain at least $\delta(K_n)=n-1$ by \cref{dumbLower}, showing the gap is arbitrarily large.

At this point, we can show that the game of Cat Herding is monotonic with respect to subgraphs.

\begin{theorem}\label{subgraphMonotone}
    If $H$ is a subgraph of $G$, then $\cat(H)\leq \cat(G)$. Further, if $v\in V(H)$, then $\cat(H,v)\leq \cat(G,v)$.
\end{theorem}
\begin{proof}
    This follows immediately from \cref{noPass}. We show that a game in $G$ scores at least as high as the corresponding game in $H$. The cat's strategy is to simply play in $H$, and without loss of generality, the herder will never cut outside of $H$ since it is never as effective as cutting $e\in E(H)$. As such, the game will proceed as normal, and the cat will score $\cat(G,v)\geq \cat(H,v)$. The cat may have better strategies in $G$ than tying their hand behind their back, justifying the inequality.
\end{proof}
It follows from \Cref{subgraphMonotone} that the cat may at any point on its turn decide to do a virtual cut of an edge \emph{or} vertex, and decide to play on a subgraph of $G$. Similarly, it suffices to only consider connected graphs.

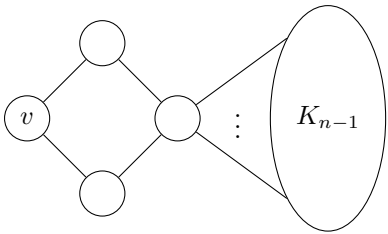
\begin{figure}
    \centering
    \begin{tikzpicture}
        \node[main node] (v) at (-1,0) {$v$};
        \node[main node] (u) at (0,1) {};
        \node[main node] (w) at (0,-1) {};
        \node[main node] (x) at (1,0) {};
        \node[ellipse, minimum height=3cm, draw] (e) at (3,0) {$K_{n-1}$};
        \draw (v) -- (u) -- (x) -- (w) -- (v);
        \draw (e.north west) -- (x) -- (e.south west);
        \node at (1.8,0) {$\vdots$};
    \end{tikzpicture}
    \caption{Values $\cat(G,v),\cat(G-e,v)$ may have arbitrarily large difference.}
    \label{fig:stallArbitrarilyBad}
\end{figure}

%% file: Sections/families.tex
\section{Graph families}
We will now analyze some simple graph families' cat numbers, namely paths, cycles, stars, wheels, and planar graphs (asymptotically).

\input{Sections/Families/paths}
\input{Sections/Families/pathsToTreesFail}
\input{Sections/Families/cycles}
\input{Sections/Families/stars}
\input{Sections/Families/wheels}
\input{Sections/Families/planar}

\newcommand{\cutwidth}[0]{\mathrm{cutwidth}}
The notion of edge separators leads to another natural bound in terms of cut-width. We recall the definition of cut-width. Let $G$ be a graph and $v_1,\dots, v_n$ be a indexing (ordering) $I$ of the vertices of $G$. For this indexing $I$, set $wd(G,I)=\max_i|[\{v_1,\dots, v_i\},\{v_{i+1},\dots, v_n\}]|$, the number of edges required to separate the first $i$ vertices from the rest. Amongst all indexings $I$, we have $\cutwidth(G)=\min_Iwd(G,I)$. We can now state our theorem.
\begin{theorem}
    If $G$ is a graph with $n$ vertices, then $\cat(G)\leq \cutwidth(G)\lceil\log_2n\rceil$.
\end{theorem}
\begin{proof}
    Let $v_1,\dots,v_n$ be an indexing that minimizes width, and thus all cuts separating $v_1,\dots,v_i$ from $v_{i+1},\dots, v_n$ are at most $\cutwidth(G)$. The herder's strategy is to play as if on a path $v_1,\dots, v_n$, and when they would delete edge $v_iv_{i+1}$, they instead cut all edges from vertices in $\{v_1,\dots, v_i\}$ to vertices from $\{v_{i+1},\dots,v_n\}$. Note that this herder strategy on paths does not ask the cat's specific location, simply which component it is in. It is thus not relevant if the cat is able to end their turn in the same spot (in the corresponding path game). Since each `path' cut costs at most $\cutwidth(G)$ real cuts, and at most $\lceil\log_2n\rceil$ `path' cuts must be made, we obtain the bound.
\end{proof}

We next pursue the cat numbers of complete graphs as the main result, but the argument will require its own section.

%% file: Sections/Families/paths.tex
\begin{theorem}\label{thm:PathHerderNumber}
    For all $n\geq 2$, $\cut(P_n)=\lceil\log_2n\rceil$.
\end{theorem}
    The proof of Theorem~\ref{thm:PathHerderNumber} follows immediately from the following two lemmas.
\begin{lemma}
    For all $n\geq 2$, $\cat(P_n)\leq \lceil \log_2n\rceil$. 
\end{lemma}
\begin{proof}
    We show this result by induction on $n$. For $n=2$, it is easy to see that $P_2$ has one edge and so after a single cut, the cat is isolated, and thus $\cat(P_2)=1=\lceil\log_22\rceil$.

    Otherwise, suppose $n\geq 3$ and for all $2\leq n'<n$, we have that $\cat(P_n)\leq \lceil\log_2n\rceil$. Let the cat choose their starting location. 
    
    \par If $n$ is even, then there is a central edge, which the herder cuts, leaving the cat on $P_{n/2}$. Since $n\geq 4$, we know that $2\leq n/2<n$, so that induction applies, and inductively the herder can isolate the cat in $1+\cat(P_{n/2})\leq 1+\lceil\log_2n/2\rceil=\lceil\log_2n\rceil$.

    \par If $n$ is odd and vertex $v$ is the center vertex, then the herder cuts an edge incident to $v$ (either edge if the cat is on $v$, or the edge further from the cat otherwise). This leaves the cat on $P_{(n+1)/2}$. Noting that $n\geq 3$ implies that $2\leq (n+1)/2<n$, the herder can inductively isolate the cat in at most $\lceil\log_2(n+1)/2\rceil$ moves. Overall, the herder isolates the cat in fewer than $1+\cat(P_{(n+1)/2})\leq 1+\lceil\log_2\frac{n+1}{2}\rceil=\lceil\log_2(n+1)\rceil=\lceil\log_2n\rceil$. We obtain the final inequality by noting that ($\lceil\log_2(n+1)\rceil>\lceil\log_2n\rceil$ with $n$ an integer) implies that ($n=2^k$), violating $n$ odd.
\end{proof}

\begin{lemma}
    For all $n\geq 2$, $\cat(P_n)\geq \lceil\log_2n\rceil$.
\end{lemma}
\begin{proof}
    Let $n\geq 2$. It suffices to show that there exists a cat strategy forcing $\lceil\log_2n\rceil$ cuts on $P_n$. Let $C$ be the cat strategy where the cat always moves to a middle-most unoccupied vertex in the path. If the cat is on an even length path and starts in the middle (i.e. they cannot move to a middle-most unoccupied vertex), they will make a virtual cut of a leaf, and play on the resulting path. Note that this virtual cut makes it an odd length path with two middle-most vertices, and thus the normal strategy will apply.
    
    
    Note that inductively, every cat turn ends with the cat in the middle of a path of length $n$ (from the cat's perspective). We will show that regardless of which edge is deleted, the cat will end on a path of length at least $\frac{n}{2}$. For the base cases, when $n=1$, the cat loses immediately, and when $n=2$, the cat indeed moves to the other vertex and lasts $\lceil n/2\rceil = 1$ cuts. Inductively, let the cat be on a middle vertex of $P_n$ for $n\geq 3$. Label the vertices in order $v_1,\dots, v_n$. We split on the parity of $n$.
    
    If $n=2k$, then either of $v_k,v_{k+1}$ is a middle vertex that the cat may be on. Without loss of generality, say $v_k$. Suppose edge $\{v_i,v_{i+1}\}$ is deleted. If $i<k$, then vertices $v_k,\dots, v_{2k}$ are still connected and the cat can freely move to the middle of a path of at least $k+1$ vertices. Otherwise, if $i\geq k$ and $i\neq 2k-1$, then the cat has vertices $v_1,\dots, v_k$, and can freely move to the middle of a path of at least $k$ vertices. The final case is when edge $\{v_{2k-1},v_{2k}\}$ is deleted. Then the cat will (virtually) delete the edge $\{v_{2k-2},v_{2k-1}\}$, and move to $v_{k-1}$ as the center of $P_{2k-2}$. Thus the cat has scored at least $\lceil n/2\rceil = k$ vertices in any of these cases.
    
    If $n=2k+1$, then the cat starts on $v_{k+1}$. Suppose the herder deletes edge $\{v_i,v_{i+1}\}$. If $i< k$, then the cat will have at least $k+1$ vertices ($v_{k+1},\dots, v_{2k+1}$) to move on, and the cat is not presently in the center of the resulting path. If $i>k$, then the cat has at least $k+1$ vertices ($v_1,\dots, v_{k+1}$) to move on. In either case, since $v_k$ was where we started and some edge was deleted, $v_k$ is no longer the middle vertex, and the cat does have some new center vertex to move to. Thus the cat may score at least $\lceil n/2\rceil = k+1$ vertices in either of these cases.
    
    This completes the argument that the cat may score at least $\lceil n/2\rceil$ vertices after a single cut following a binary search strategy. It follows inductively that binary search scores at least $\lceil \log_2n\rceil$ cuts. Since a cat strategy exists that achieves the desired bound, we get that $\cat(P_n)\geq \lceil \log_2n\rceil$ by \Cref{stratRules}.
\end{proof}

%% file: Sections/Families/pathsToTreesFail.tex
Suppose $G$ contains a subgraph $H$, and $H$ belongs to a graph family with a known analysis. Then $\cat(H,v)\leq \cat(G,v)$ provides a lower bound for some vertices $v\in V(G)$. In particular, we may consider all leaf-to-leaf paths in trees, and obtain a lower bound for trees' cat numbers. This is not generally tight; see a minimal counterexample in \Cref{fig:treeBoundNotTight}.

\begin{figure}
    \centering
    \begin{subfigure}[c]{\textwidth}
        \centering
        \begin{tikzpicture}
            \node[main node] (a) at (0,1) {$1$};
            \node[main node] (b4) at (0,0) {$4$};
            \node[main node] (b5) at (1,0) {$3$};
            \node[main node] (b6) at (2,0) {$2$};
            \node[main node] (b7) at (3,0) {$1$};
            \node[main node] (b3) at (-1,0) {$3$};
            \node[main node] (b2) at (-2,0) {$2$};
            \node[main node] (b1) at (-3,0) {$1$};
            \draw (b1) -- (b2) -- (b3) -- (b4) -- (b5) -- (b6) -- (b7);
            \draw (b4) -- (a);
        \end{tikzpicture}
        \caption{Actual cat values}
        \label{fig:treeLowBoundNotTightA}
    \end{subfigure}
    \begin{subfigure}[c]{\textwidth}
        \centering
        \begin{tikzpicture}
            \node[main node] (a) at (0,1) {$1$};
            \node[main node] (b4) at (0,0) {$3$};
            \node[main node] (b5) at (1,0) {$3$};
            \node[main node] (b6) at (2,0) {$2$};
            \node[main node] (b7) at (3,0) {$1$};
            \node[main node] (b3) at (-1,0) {$3$};
            \node[main node] (b2) at (-2,0) {$2$};
            \node[main node] (b1) at (-3,0) {$1$};
            \draw (b1) -- (b2) -- (b3) -- (b4) -- (b5) -- (b6) -- (b7);
            \draw (b4) -- (a);
        \end{tikzpicture}
        \caption{Lower bounds from all paths}
        \label{fig:treeLowBoundNotTightB}
    \end{subfigure}
    \caption{An example of a tree, where each vertex contains (a) the cat number of that vertex in the tree, and (b) the lower bound of that vertex in some leaf-to-leaf subgraph containing that vertex. The lower bounds on the herder number from leaf-to-leaf paths are not tight.}
    \label{fig:treeBoundNotTight}
\end{figure}
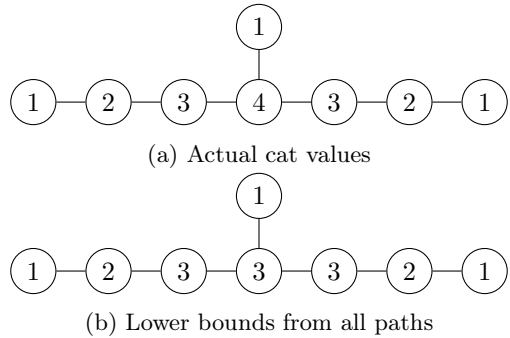

%% file: Sections/Families/cycles.tex
\begin{theorem}
    If $k\geq 2$, then $\cat(C_{2k+1})=\cat(C_{2k})=\lceil \log_22k\rceil+1$. Similarly, $\cat(C_3)=2$.
\end{theorem}
\begin{proof}
    First, odd cycles have $\cat(C_{2k+1})\leq \lceil \log_22k\rceil+1$. To show this, we need a herder strategy. Set $H$ to be the strategy as follows: for the first deletion, delete the edge diametrically opposed to the cat (existence follows by $C_n$ being an odd cycle). The cat is now on the center vertex of a path $P_{2k+1}$, and must now move off of it. The herder deletes the first edge from the path that the cat just moved along. This locks the cat into $P_k$ in $2$ cuts, and continuing the game will take at most $\cat(P_k)=\lceil \log_2k\rceil$ cuts. Overall, then, the herder takes at most $2+\lceil\log_2k\rceil=1+\lceil 1+\log_2k\rceil = 1+\lceil\log_22k\rceil$ cuts.
    
    We next show that $\cat(C_{2k+1})\geq \lceil \log_22k\rceil+1$. To show this, we need a cat strategy. For the first move, choose any vertex $v\in V(C_{2k+1})$. After this, the cat plays the path strategy on the resulting path. The only change from the previous analysis is that the assumption that we always end our turns on the center of a path is no longer true. Rather, after the first cut, we will be somewhere on the path. If we are not in the center, we may then move to the center, and then play as normal on $P_{2k+1}$, and we end with a score of $\lceil \log_22k+1\rceil +1$ cuts, which is at least the desired bound of $\lceil\log_22k\rceil +1$. If we are in the center, though, we follow our previous strategy and do a virtual cut of a leaf, leaving us on one of the two center vertices of $P_{2k}$. Then we will move to the other center vertex and continue play on $P_{2k}$. In this case, we score at least $\cat(P_{2k})+1=\lceil\log_22k\rceil+1$ cuts, and thus we establish the desired bound.
    

    Next, we show that even cycles have $\cat(C_{2k})=\lceil\log_22k\rceil+1$. Consider an optimal game. Let the cat start at $v\in V(C_{2k})$ (since cycles are vertex-transitive this is without loss of generality). After the first optimal cut, the cat may play their already found optimal path strategy on $P_{2k}$ (since there are two middle vertices, at least one must be possible to move to). The length of this optimal game is then $\cat(C_{2k})=1+\cat(P_{2k})=\lceil\log_22k\rceil + 1$. 
\end{proof}

%% file: Sections/Families/stars.tex
Let $S_i$ denote the star graph, which has a vertex with $i$ additional leaves appended to this vertex. That is, $S_i=K_{1,i}$.
\begin{lemma}
    If $n\geq 2$, then $\cat(S_n)=2$. For small stars, $\cat(S_0)=\cat(K_1)=0$ and $\cat(S_1)=\cat(P_2)=1$. 
\end{lemma}
    
\begin{proof}
    For the herder strategy, if the cat starts on a leaf, cut the edge incident and win in one cut. Otherwise, cut an arbitrary edge. Then the cat will move to a leaf and the herder will cut the edge just traversed, winning in at most $2$ cuts.
    
    For the cat, start in the middle, and move to any available leaf on the second move. This shows that $\cat(S_n)=2$.
\end{proof}

\begin{lemma}
    If an optimal game of Cat Herding is played on $G$ with $\cat(G)\geq 2$, then the cat's component after the penultimate herder turn is a star, and the cat's unique optimal move is to move to the center of the star.
\end{lemma}
\begin{proof}
    By definition, the game ends when the cat is isolated. Backtracking one turn, the cat must have been on a leaf, and the herder deleted the edge incident. Since a move to a leaf guarantees loss after a single more turn, and any other move cannot result in the game ending in a single turn, the cat will not move to a leaf unless necessary. It follows that the cat must have had no other legal moves than to move to a leaf. Thus the cat is in the center of a star.
\end{proof}

Note that paths have $\cat(T)\approx \log |V(G)|$, while stars have $\cat(T)\leq 2$. We have found that the cat number of trees can vary widely.

%% file: Sections/Families/wheels.tex
Recall $W_n$ is a wheel with $n-1$ vertices $v_1,\dots, v_{n-1}$ on a cycle and one vertex $v_0$ connected to all other vertices. 
\begin{theorem}
    For $n\geq 3$, $\cat(W_n)=n+1$.
\end{theorem}
\begin{proof}
    For the herder strategy herding the cat in at most $n+1$ cuts, cut all $n-1$ edges in the cycle. The resulting graph is a star, which can be won in $2$ cuts by making a stalling move, the cat moves to a leaf, and the herder isolates the leaf. As a result, the herder will win in $n-1+2=n+1$ cuts.
    
    For the cat's strategy, we define a `safe zone' as any cycle containing the central vertex $v_0$ and having no internal chords. Note that the game starts with $n-1$ safe zones, one for each edge in the cycle  combined with the spokes to the center. 
    
    Note that most cuts either keep the number of safe zones constant or decreases them by one. Suppose edge $e$ is cut. If $e$ is a spoke, it may merge two safe zones (if there is at least $2$ other spokes), destroy a safe zone, or delete a leaf, all at most decreasing safe zones by $1$. If $e$ is an outer-cycle edge, at most one safe zone cycle or leaf is deleted. The sole exception is on cutting the next-to-last spoke, in which case we move from $2$ to $0$ safe zones. This case will be evaluated separately. The cat's strategy is to always move to a safe zone's non-central vertex. This is generally possible since all safe zones have at least 2 edge-disjoint paths to all other safe zones (both have at least $2$ edges to the shared central vertex), and each safe zone has at least $2$ candidate vertices (cycles have at least $3$ vertices, at most one may be central). The cat can employ this strategy for $n-1$ cuts at least until all safe zones are cut, or in the exceptional case have $n-2$ spokes cut to remove all safe zones.
    
    When the last safe zone is cut, in the exceptional case where $n-2$ spokes are cut, we have $C_n$ with a pendant leaf, and the cat has scored $n-2$ already. The cat will score at least $3$ on $C_n$ if $n>3$. The remaining case $C_3$ with a tail can be computationally evaluated (cat moves on $C_3$ non-adjacent to the tail, then after the next cut can move to the center of a $P_3$, scoring $2$ more) to have score at least 3, thus the cat will score $n+1$.
    
    In the more general case, the herder has made at least $n-1$ cuts to remove all safe zones, so that there are no cycles containing the center vertex, it is the cat's turn and the cat is not in the center vertex. Further, the cat was in a safe zone, and thus is on a path of at least $3$ (from the cycle the cat was previously on) vertices. If the cat is not on the center of the path or the path has more than $3$ vertices, then they can score at least $2$ more by moving to the center-most available vertex, totalling $n+1$ cuts. If the cat is in the center of three vertices remaining, then the cat has a spoke and an outer-cycle edge adjacent to them. If all other spokes have been cut, that means that $n-2$ spokes in total have been cut, and we are in the previous case. Thus there is a $P_3$ formed by the spokes that is reachable to the cat, scoring $2$ more by the cat moving to the center. Overall, the cat can score $n+1$ again.
        
    Collectively, this provides a cat strategy that scores at least $n+1$ for a wheel on $n$ vertices. Combined with the herder strategy of herding in at most $n+1$ cuts, this shows that $\cat(W_n)=n+1$.
\end{proof}

%% file: Sections/Families/planar.tex
We now look at the asymptotics of Cat Herding on planar graphs. We next show that planar graphs $G$ with $n$ vertices and maximum degree $\Delta$ have $\cat(G)$ asymptotically growing with $O(\sqrt{\Delta n})$, and this is asymptotically tight. We refer to \cite{DIKS1993258} for the main idea for these arguments, based on edge separators.

An \emph{edge separator} of a graph $G$ is a partition of $V(G)$ into $A,B$, such that $|A|\leq \frac{2}{3}n,|B|\leq \frac{2}{3}n$. We say that $|[A,B]|$ is the \emph{size} of the edge separator, where $[A,B]=\{\{a,b\}\in E(G)|a\in A,b\in B\}$.

Theorem 2.2 of \cite{DIKS1993258} informs us that a planar graph $G$ with maximum degree $\Delta$ and $n$ vertices has an \emph{edge separator} of size $\leq 2\sqrt{2\Delta n}$.

\begin{theorem}\label{thm:planarUpper}
    If $G$ is a planar graph on $n$ vertices of maximum degree $\Delta$, then $$\cat(G)\leq 2\sqrt{2}(3 + \sqrt{6})\left(\sqrt{\Delta n}\right)\leq 15.42\sqrt{\Delta n}$$
\end{theorem}
\begin{proof}
    Let $G$ be a planar graph on $n$ vertices with maximum degree $\Delta$. We proceed by induction on $n$. As a base case, if $n=1$, then $\Delta=\cat(G)=0$, and we are done. Otherwise, we apply Theorem 2.2 of \cite{DIKS1993258} to get an edge separator $V(G)=A\cup B$ of size $\leq 2\sqrt{2\Delta n}$. Then the herder's strategy is to cut $[A,B]$, and then play on a subgraph induced by either $A$ or $B$, each of which contain at most $\frac{2n}{3}$ vertices. Without loss of generality, assume it is component $A$. The overall cost is then computable by the following.
    \begin{align*}
        \cat(G)
        &\leq |[A,B]|+\cat(A)\\
        &\leq 2\sqrt{2\Delta n}+2\sqrt{2} (3 + \sqrt{6})\left(\sqrt{\Delta \frac{2}{3}n}\right)\\
        &= 2\sqrt{2} (3 + \sqrt{6})\left(\sqrt{\Delta n}\right)
    \end{align*} 
    We thus obtain the desired bound.
\end{proof}

It is known that replacing each edge of $P_p\square P_p$ with $k$ copies of $P_3$ results in a graph with edge separator of size at least $\sqrt{\Delta n}/9=\sqrt{4k(p^2+2kp(p-1))}/9$; see \cite{DIKS1993258}. We show that this graph family has cat number asymptotically $\sqrt{\Delta n}$. We thus know exactly the asymptotic cat behaviour on planar graphs.

Let $p,k\in \mathbb{N}$, with $k\geq 1, p\geq 3$. Set $H=P_p\square P_p$. Replace every edge with $k$ edges between the same vertices. Subdivide every edge with one new vertex. Call the resulting graph $Gr(p,k)$.

It should be clear that $P_p\square P_p$ has $p^2$ vertices and $2p(p-1)$ edges. On following the construction of $Gr(p,k)$ it should be clear that each edge results in $k$ new vertices. Thus in total $|V(Gr(p,k))|=p^2+2kp(p-1)$. Also note that $\Delta(Gr(p,k))=4k$ when $p >2$.

\begin{lemma}
    If $k\geq 1,p\geq 3$, then $\cat(Gr(p,k))\geq kp$.
\end{lemma}
\begin{proof}
    The cat's strategy is to start anywhere on a vertex $(u,v)$ from $P_p\square P_p$. For all $u,v$, let $H(u,v)$ be the graph induced by all vertices with $u$ as first coordinate or $v$ as second from $P_p\square P_p$, as well as all vertices created by subdividing edges in $H(u,v)$, so that $H(u,v)$ has $k$ edge-disjoint paths between any pair of vertices from $P_p\square P_p$ (at least initially). We argue that after the first $kp-1$ cuts, some vertex $(x,y)$ exists such that: 1) all $v\in V(H(x,y))$ from $P_p
    \square P_p$ are connected, and 2) there exists an $(x,y),(u,v)$ path. We will call such $H(u,v)$ \emph{intact}.

    Note that for the first $kp-1$ cuts (when the cat needs a response), at least one row and at least one column must both still have paths between all vertices (as it takes $k$ cuts to fully disconnect any given $P_p$, and removing all $p$ in either rows/columns must take at least $kp$ cuts). Thus an intact $H(u,v)$ exists. The cat's strategy is to move to $(u,v)$. Note that since $H$ initially contained a full vertical/horizontal set of $p$ grid vertices, with only one cut, some row or column is still fully connected. This must meet the fully intact row/column pair that must exist, so the cat can move there.

    The sole exceptional case is when all fully intact rows/columns are the same as the one the cat is currently in. In this case, there are $(p-1)$ other rows/columns that have been fully cut, thus costing at least $2k(p-1)=kp+k(p-2)$ cuts. Since we restrict $p\geq 3$, this takes strictly more cuts to reach than $kp$.
\end{proof}
\begin{theorem}\label{thm:planarLower}
    For $p \geq 3$, the family $Gr(p,k)$ have $\Delta,n$ arbitrarily large, and $\cat(Gr(p,k))\geq \frac{\sqrt{3}}{6}\sqrt{\Delta n}>0.288\sqrt{\Delta n}$.
\end{theorem}
\begin{proof}
    All we need to do is show that $kp\geq \frac{\sqrt{3}}{6}\sqrt{\Delta n}$. Noting everything is positive, it will be simpler to show that $k^2p^2\geq \frac{1}{12}\Delta n$.
    \begin{align*}
        \frac{1}{12}\Delta n
        &=\frac{1}{12}4k(p^2+2kp(p-1))\\
        &=k^2p^2\left(\frac{1}{3}\left(\frac{1}{k}+2-\frac{2}{p}\right)\right)\\
        &\leq k^2p^2\left(\frac{1}{3}\left(\frac{1}{k}+2\right)\right)\\
        &\leq k^2p^2\left(\frac{1}{3}\left(3\right)\right)\\
        &=k^2p^2
    \end{align*}
    We again have obtained the desired bound.
\end{proof}
\begin{corollary}
    Planar graphs have $\cat(G)=O(\sqrt{\Delta n})$, and this is tight.
\end{corollary}
\begin{proof}
    The proof is immediate from the previous bounds in \Cref{thm:planarUpper} and \Cref{thm:planarLower}.
\end{proof}

%% file: Sections/complete.tex
\section{Complete Graphs}\label{sec:Complete}
\input{Sections/Complete/prelude}
\subsection{General Approach}
\input{Sections/Complete/CH}
\subsection{Herder Strategy}
\input{Sections/Complete/HCisC}
\subsection{Full Cat Strategy}
\input{Sections/Complete/fullC}
\subsection{Binary Search Herder Recurrence}
\input{Sections/Complete/recurrence}
\subsection{Herder Strategy is optimal against Cat Strategy}
\input{Sections/Complete/HCisH}

%% file: Sections/Complete/prelude.tex
In paths, the herder essentially uses binary search to isolate the cat, cutting the graph in half each time. We aim to argue an analogous herder strategy is optimal for complete graphs. Namely, we show that an optimal strategy on complete graphs for the herder is to cut $K_n$ into $K_{\lceil n/2\rceil}$ and $K_{\lfloor n/2\rfloor}$, and repeat until a base case of $K_2=P_2$ or $K_3=C_3$ is reached. We will discover that there are some cases where the binary search strategy is not the unique optimal herder strategy, although it is certainly one of the simplest to describe, and no strategy can surpass it in performance.

The overall objective is to show the following results, spread across \Cref{sec:Complete}.
\begin{restatable}[]{theorem}{comCatThm}
\label{thm:mainComplete}
    If $c_2=1$, $c_3=2$, and $c_n=\lfloor \frac{n}{2}\rfloor\lceil \frac{n}{2}\rceil+c_{\lceil\frac{n}{2}\rceil}$ for $n> 3$, then for $n\geq 2$, $\cat(K_n)=c_n$, $\frac{n^2}{3}-1\leq \cat(K_n)\leq \frac{n^2+n}{3}$, and 
    \[\cat(K_n)=\frac{n^2-n}{2}-\lfloor \log_2(n-1)\rfloor-\sum_{j=2}^{\lfloor \frac{n-2}{2}\rfloor}j\left(\left\lfloor 1+\log_2\frac{n-1}{2j+1}\right\rfloor\right).\]
\end{restatable}

This provides a concise recurrence representing binary search, asymptotic bounds on the $\cat$ number, and (surprisingly) provides an exact closed formula for $\cat(K_n)$.

%% file: Sections/Complete/CH.tex
We first observe that Cat Herding is an instance of a minimax game.

\begin{definition}
    A \emph{minimax game} is a finite collection of game states, $S\times \{L,R\}$, where $S$ is the current state, and $L,R$ denotes which player has the next move. In addition, each state $s\in S$ can be represented as $\{L|R\}$, where $L,R$ are the collections of legal moves for each player in that state. Some game states $s\in S$ are considered terminal, in which case there is an associated score from a totally ordered set of scores (we will use $\mathbb{N}$ for this context). One player's objective is to maximize score, while the other's is to minimize it.
\end{definition}

We note that the game of Cat Herding is a minimax game where $S=\{G\}\cup \bigcup_{H\subseteq G}(H\times V(G))$, for subgraphs $H$ of $G$ obtained by edge deletion, and allowing position $G$ without a cat token position as the initial setup.

Note that by maintaining a count of the number of edges the herder has deleted, the game of Cat Herding satisfies the conditions of a minimax game, with a game tree of height at most $|E(G)|$.

\begin{definition}
    A \textbf{sub-position} of a minimax game is any minimax game that is reachable from the starting position. Let $\mathcal{S}$ be the set of all sub-positions. 
    An \textbf{$L$-strategy} for a minimax game is a function from $\ell:\mathcal{S}\rightarrow \mathcal{S}$ such that if $\ell(\{L|R\})=s$, then $s\in L$, so that $\ell$ maps positions to legal $L$ moves. The \textbf{$R$-strategies} are defined likewise.
\end{definition}

We view $L$-,$R$-strategies essentially as rules that decide what move to make from a given position.

Suppose $G$ is a minimax game. The left player ($L$) tries to maximize the score and the right player ($R$) tries to minimize it. We will use $L,R$ to denote strategies for left/right, respectively, functions from non-terminal states to child states (moves) describing what $L,R$ will do in each game state. Denote $s(G,L,R)$ to be the score obtained when $L$ plays $R$ at game $G$ according to their strategies. When we omit $L,R$ from $s(G,L,R)$, we let the omitted strategy be any optimal strategy for that player. An $L$-strategy is considered optimal when for all $R$, $s(G,L,R)\leq s(G,L)$, and an $R$-strategy is considered optimal when for all $L$, $s(G,R)\leq s(G,L,R)$. Because minimax games have finitely many states, from any position $s\in S\times \{L,R\}$, one can take the minimum or maximum (respectively, for $L,R$) of $s$'s direct children sub-positions' values in order to compute the value of $s$. Doing so recursively assigns all states a value.  Optimal strategies must exist then, as from any state, there must be a move preserving score.

\par Note then that $s(G)=s(G,L)=s(G,R)$ means that $L$ scores as well as any optimal left-strategy, and similarly for $R$, meaning $L,R$ are optimal as well.
\begin{lemma}\label{lrScore}
    If $G$ is a scoring game, $L,R$ are strategies for the $2$ players such that $L$ is optimal against $R$, and $R$ is optimal against $L$, then both $L,R$ are optimal against any strategy.
\end{lemma}
\begin{proof}
    Note this is well-defined as all strategies will be deterministic. Since $L$ is optimal against $R$, we find that $s(G,R)=s(G,L,R)$. Similarly, $R$ optimal against $L$ implies $s(G,L)=s(G,L,R)$. Since left is trying to maximize score and right minimize it, $s(G,L)\leq s(G)\leq s(G,R)$. We then find that

    \[s(G,L,R)=s(G,L)\leq s(G)\leq s(G,R)=s(G,L,R)\]

    If any inequality is strict, we reach a contradiction $s(G,L,R)<s(G,L,R)$, and thus $s(G,L)=s(G)=s(G,R)$, and both $L,R$ are optimal strategies.
\end{proof}

%% file: Sections/Complete/HCisC.tex
Let $C$ be any strategy that guarantees that when the binary search herder disconnects the graph, the cat is on the component with the most vertices. We will later specify a full strategy $C$ that satisfies this condition.

We now specify a herder strategy which will be optimal.
\begin{algorithm}
\caption{Optimal Herder Strategy $H$}\label{alg:herderStrat}
\begin{algorithmic}[1]
\Require $G=K_n$ for some $n\in \mathbb{N}$
\Ensure Output series of cuts that are an optimal sequence to isolate the cat.
\State If $G=K_2,K_3$, herder plays $G=P_2,C_3$ strategy.
\State Partition $V(K_n)$ into $V(K_n)=S\cup T$ where $S\cap T=\emptyset, ||S|-|T||\leq 1$.
\State Over the next $|S||T|$ turns, delete all edges from $S$ to $T$ (the edge set $[S,T]$).
\State Repeat on whichever component the cat is on.
\end{algorithmic}
\end{algorithm}

\begin{lemma}\label{KnHCisC}
    Any optimal strategy $C'$ against $H$ on $K_n$ may be replaced by $C$ without decreasing the score; that is, $\cat(K_n,H,C)=\cat(K_n,H)$.
\end{lemma}
\begin{proof}
    Let $C'$ be an optimal strategy in Cat Herding on $K_n$ against the herder strategy $H$. Consider games of $H$ vs $C$ and $H$ vs $C'$ played in parallel. Note that since the herder strategy takes no input from where the cat is excepting when the graph is cut in half and in $K_2,K_3$ endgames, we only need to consider when $C,C'$ differ in play on these turns. For $K_2,K_3$ endgames, we note that these are small cases, and one can easily check that $C$ performs optimally in these cases. Otherwise, consider the turns when a herder executing binary search fully disconnects the graph, say into $K_n,K_m$ with $n\geq m$.
    
    Note $n\geq m$ implies $\cat(K_n)\geq \cat(K_m)$ by Theorem~\ref{subgraphMonotone}. If $n>m$, then a game on $n$ may be played by voluntarily restricting cat play to $K_m$, and on being isolated in the subgraph, there must be one more edge to an exterior vertex (or else extra cuts have been made). Thus $\cat(K_n)>\cat(K_m)$, and $\cat$ on $K_n$ is a strictly increasing function. Thus when $C'$ causes the binary herder to disconnect the graph, the two components are equal in size, or $C'$ has the cat moving to the larger component. Similarly, $C$ by construction also satisfies this property. We do not care specifically which vertex is targeted in $K_n$ as the herder will only query component, not precise location, and so both cat strategies achieve the same score. Finally noting the score from induction in $K_n$ is the same in both $C,C'$, this completes the inductive argument.
\end{proof}

%% file: Sections/Complete/fullC.tex
The cat's strategy pays attention to the structure of maximal $2$-edge-connected components in the first graph $G$ that is not $2$-edge-connected. Such graphs will be referred to as `critical graphs'.
\begin{definition}
    For any given $K_n$ and any herder strategy $H'$, we let $G$ be the first critical graph (not $2$-edge-connected) reached in play against $C$. We also let $(\lambda_i)_{i=1}^k$ denote the structure of $G$, where each $\lambda_i$ corresponds to a $2$-edge-connected component of $\lambda_i$ vertices, partitioning $V(G)$, so that $\sum_{i=1}^k\lambda_i=n$, and $\lambda_i\geq \lambda_j$ whenever $i\geq j$. This also leaves us a common convention of $|\lambda|=k$, and we will generally use $k$ for the size of $\lambda$.
\end{definition}
We consider the example of Cat Herding in $K_{39}$. During play on $K_{39}$, the resulting graph after each herder cut is result in $2$-edge-connected, until cutting some edge $e$ results in more than one $2$-edge-connected component. On edge $e$ being deleted, there are $2$-edge-connected components of size $12,1,5,13,8$, each connected by a bridge to the previous one, in that order. In this case, we would write $\lambda=(13,12,8,5,1)=(\lambda_5,\lambda_4,\lambda_3,\lambda_2,\lambda_1)$, denoting that these are the sizes of the $2$-edge-connected components, and we sort them by order for convenience. Also note here that $\lambda_5+\lambda_4+\lambda_3+\lambda_2+\lambda_1=13+12+8+5+1=39$, so that $\lambda$ may also be viewed as an integer partition of $39=|V(K_{39})|$. Note that as a consequence of the definition, the integer $2$ may never be part of a partition $\lambda$ that represents a critical graph, as $K_2$ is not $2$-edge-connected.

These $\lambda_i$ induce integer partitions (of size $k$) which are convenient to reason over. With our notation, we now specify the full cat strategy, using a vertex labelling to keep the strategy deterministic. We break the cat's strategy up into multiple sub-strategies in \Cref{alg:nearlyBin,alg:nSmall}, with the full strategy given in \Cref{alg:general}.

\begin{algorithm}
\caption{Cat Strategy on small graphs}\label{alg:nSmall}
\begin{algorithmic}[1]
\Require $n\leq 9$, and $\lambda_k=3$ or $\lambda=(6,1)$)
\Ensure Cat makes an optimal move assuming play originated on $K_n$.
\If{$\lambda=(3,3,1)$}
    \State The cat moves to the center-most component.\label{lin:catMoveCenter}
\ElsIf{$\lambda_k=3$}
    \State The cat moves to a vertex of a $K_3$ such that one of the other two vertices of this $K_3$ has a pendant edge. 
    \State On the cat's next turn, if the herder does not cut within the cat's $K_3$, then the cat moves within $K_3$ and voluntarily restricts play there. Otherwise, the cat has a path within the old $K_3$ to the pendant edge (since $k\geq 2$), providing access to the center of a $P_3$ which they move to, proceeding to play on a path.
\ElsIf{$\lambda=(6,1)$}
    \State The cat temporarily voluntarily restricts their play to $K_6$ until a new critical graph's $\lambda$ is reached, at which point they return to the start of the strategy.
\EndIf
\end{algorithmic}
\end{algorithm}
In \Cref{alg:nSmall}, line \ref{lin:catMoveCenter}, such a move is allowed since the cat selects maximum degree vertices on the previous turn, and if $K_1$ is in center component, it has $\deg(v)=2$, while some vertex of each of the $K_3$ has $\deg(v)=3$, which must have been where the cat was. Note also that this means that in some partitions, such as $(3,3,3,1)$ with the single vertex as a center node branching to each of the three triangles, the cat will play (locally) sub-optimally in order to guarantee this condition. However, this locally sub-optimal play remains globally optimal as the cat's strategy sacrifices less score than they gained to reach the substructure. These sorts of locally sub-optimal decisions likely occur elsewhere, but as we assume optimal play from both players, this does not affect our proof.

\begin{algorithm}
\caption{Cat's Strategy when herder behaves close to binary search}\label{alg:nearlyBin}
\begin{algorithmic}[1]
\Require $\lambda_k-\lambda_{k-1}\leq 1$, $\lambda_k> 1$
\Ensure Cat makes an optimal move assuming play originated on $K_n$.

\If{$\lambda_k=\lambda_{k-1}$}
            \If{any cuts are made within $C_{k-1},C_k$}
                \State The cat moves to the component $C_k-1$ or $C_k$ with fewer edge cuts and voluntarily restricts their play there henceforth.
            \Else
                \State The cat moves anywhere in component $C_k$. 
                \State After any next cut, at least one component of size $\lambda_k$ has no cuts and is reachable from the cat's current location, and the cat moves there, voluntarily restricting play to that complete graph thereafter. 
            \EndIf
        \Else 
            \If{any cuts are made within $C_{k-1}$}
                \State The cat voluntarily restricts their play to $C_k$, and moves anywhere.
            \Else
                \State The cat counts any current cuts in $C_{k}$ towards the `next' $\lambda_{k-1}$ cuts that are considered. 
                \If{any of the next $\lambda_{k-1}$ cuts are outside of $C_{k}$}
                    \State The cat voluntarily restricts their play to $C_{k}$ henceforth.
                \Else
                    \State For the next $\lambda_{k-1}-2$ cuts made, the cat moves anywhere in $C_{k}$. 
                    \State For the next $2$ cuts made, move anywhere in $C_{k-1}$. \label{lin:nextTwo}
                \EndIf
            \EndIf
        \EndIf
\end{algorithmic}
\end{algorithm}
In \Cref{alg:nearlyBin}, line \ref{lin:nextTwo}, we note that $K_{\lambda_k}$ is $\lambda_{k-1}$-edge-connected. Therefore this plan for the next two cuts is always possible, as all moves within $K_{\lambda_k}$ are before $\lambda_{k-1}$ cuts have been made in $K_{\lambda_k}$.

\begin{algorithm}
\caption{Optimal Cat Strategy $C$}\label{alg:general}
\begin{algorithmic}[1]
\Require $l:V(G)\rightarrow [1,\dots, n]$ be a bijection labelling every vertex with a number.
\Ensure The cat makes an optimal move assuming play originated on $K_n$.
\If{$G$ is $2$-edge connected}
    \State Move to the vertex $v$ with maximal $\deg(v)$ and lowest label $l(v)$.
\Else
    \State Let $C_i$ be all $k$ maximal $2$-edge-connected subgraphs. 
    \State Index the $C_i$ so that $|V(C_i)|=\lambda_i$ and $i\leq j$ implies $\lambda_i\leq \lambda_j$. 
    \State Set $\lambda=(\lambda_k,\dots, \lambda_1)$ as the integer partition of $n$ representing the critical graph $G$.
    \If{$n\leq 9$ and ($\lambda_k=3$ or $\lambda=(6,1)$)}
        \State The cat executes \Cref{alg:nSmall}.
    \EndIf
    \If{$\lambda_k=1$}
        \State The cat plays known optimal path strategy.
    \ElsIf{$\lambda=(3(2^m),3(2^m),1)$}
        \State The cat moves to the center-most component. After the next cut, the cat voluntarily restricts their play in $K_{3(2^m)}$ with an adjacent vertex, so that the partition is now $(3(2^m),1)$.\label{alg:weirdCase}
        \State The cat further places a temporary restriction to not enter the adjacent vertex.
        \State The cat removes the temporary restriction when the partition changes from $(3(2^m),1)$.
    \ElsIf{$\lambda_k-\lambda_{k-1}\leq 1$}
        \State The cat executes \Cref{alg:nearlyBin}
    \ElsIf{$\lambda=(3(2^m),1)$}
        \State The cat follows the strategy of line 21 but does not restrict themselves to the component, rather will continue play until the partition $\lambda$ changes again.
    \Else
        \State The cat moves to the vertex $v$ that is in a maximum $2$-edge-connected subgraph with maximal degree and lowest label $l(v)$ that is legal to move to. 
        \State The cat voluntarily restricts themselves to this component.
    \EndIf
\EndIf
\end{algorithmic}
\end{algorithm}

It should be clear that when playing against the binary search herder $H$ the cat will always move in a $2$-edge-connected component; for $n\leq 9$, one can check \cref{alg:nSmall}, and for larger $n$, a trace through \cref{alg:nearlyBin} gives the desired property. Thus this full strategy of \cref{alg:general} satisfies the key property for \cref{KnHCisC}.

Henceforth, we will aim to show that any herder strategy $H'$ playing against the provided cat strategy $C$ will perform no better than binary search herder strategy $H$. In order to do this, we seek a strong understanding of how the binary search strategy $H$ performs.

%% file: Sections/Complete/recurrence.tex
\begin{definition}
    Let $(c_n)$ be a sequence defined by the following recurrence: $c_2=1$, $c_3=2$, and $c_n=\lfloor \frac{n}{2}\rfloor\lceil \frac{n}{2}\rceil+c_{\lceil\frac{n}{2}\rceil}$ for $n> 3$.
\end{definition}

Note that the herder strategy of binary search is exactly expressed by this recurrence; it encapsulates deleting all edges between a vertex partition, and then repeating on the larger part (where the cat will end up). We will benefit by some study of the sequence $(c_n)$. We will make use of the first differences sequence, $\Delta_n=c_{n+1}-c_n$ for $n\geq 2$. 

\begin{lemma}\label{lem:firstDiffRecur}
    If $n\geq 2$, then the first differences $\Delta_n=c_{n+1}-c_n$ obey the following recurrence.
    \begin{align*}
        \Delta_n=\left\{\begin{array}{cc}
            \frac{n}{2}+\Delta_{\frac{n}{2}} & n>3,n\ even \\
            \frac{n+1}{2} & n>3,n\ odd\\
            1 & n=2\\
            3 & n=3
        \end{array}\right.
    \end{align*}
\end{lemma}
\begin{proof}
    The base cases of $n=2,3$ are trivial as $c_3-c_2=2-1=1=\Delta_2$ and $c_4-c_3=5-2=3$.
    Suppose $n>3$, then we may use the recursive definition of $c_n$. If $n$ is odd, then we compute the following.
    \begin{align*}
        c(n+1)-c(n)
        &=\left\lfloor\frac{n+1}{2}\right\rfloor\left\lceil\frac{n+1}{2}\right\rceil+c_{\left\lceil \frac{n+1}{2}\right\rceil}-\left\lfloor\frac{n}{2}\right\rfloor\left\lceil\frac{n}{2}\right\rceil-c_{\left\lceil \frac{n}{2}\right\rceil}\\
        &=\left(\left(\frac{n+1}{2}\right)^2-\left(\frac{n-1}{2}\right)\left(\frac{n+1}{2}\right)\right)+0\\
        &=\frac{n+1}{2}
    \end{align*}
    If $n$ is even, then we may compute the following.
    \begin{align*}
        c(n+1)-c(n)
        &=\left\lfloor\frac{n+1}{2}\right\rfloor\left\lceil\frac{n+1}{2}\right\rceil+c_{\left\lceil \frac{n+1}{2}\right\rceil}-\left\lfloor\frac{n}{2}\right\rfloor\left\lceil\frac{n}{2}\right\rceil-c_{\left\lceil \frac{n}{2}\right\rceil}\\
        &=\left(\left(\frac{n}{2}\right)\left(\frac{n+2}{2}\right)-\left(\frac{n}{2}\right)^2\right)+\left(c_{\frac{n+2}{2}}-c_{\frac{n}{2}}\right)\\
        &=\left(\frac{n}{2}\right)+\Delta_{\frac{n}{2}}
    \end{align*}
    We have verified all of the necessary cases.
\end{proof}

\begin{corollary}\label{lem:delNBounds}
    If $n\geq 2$, then $\frac{n}{2}\leq \Delta_n\leq n$.
\end{corollary}
\begin{proof}
    This is immediate by an inductive argument and examining the recurrence for $\Delta_n$.
\end{proof}
\begin{lemma}\label{lem:firstDiffClosed}
    If $n=2^km$ with $m\geq 3$ odd, then $\Delta_n=n$ if $m=3$, and $\Delta_n=n-\frac{m-1}{2}$ when $m>3$. Further, $\Delta_{2^k}=2^k-1$.
\end{lemma}
\begin{proof}
    Let $n=2^km$ with $m\geq 3$ odd. Applying \cref{lem:firstDiffRecur} for even numbers $k$ times we find that $\Delta_n=\sum_{i=1}^k\frac{n}{2^i}+\Delta_m=n\left(1-\frac{1}{2^k}\right)+\Delta_m=n-m+\Delta_m$. If $m=3$, we find that $\Delta_n=n-3+3$. Similarly, for $m>3$, we find $\Delta_n=n-m+\frac{m+1}{2}=n+\frac{1-m}{2}$. Applying \cref{lem:firstDiffRecur} $k$ times to $n=2^k$ gives a similar formula $\sum_{i=1}^k2^{k-i}=2^k-1$.
\end{proof}

\begin{theorem}\label{thm:recurrence_closed}
    For $n\geq 2$

    \[c_n=\frac{n^2-n}{2}-\lfloor \log_2(n-1)\rfloor-\sum_{j=2}^{\lfloor \frac{n-2}{2}\rfloor}j\left(\left\lfloor 1+\log_2\frac{n-1}{2j+1}\right\rfloor\right)\]
\end{theorem}
\begin{proof}
    This can be found by taking the telescoping sum $c_n=c_n-c_{n-1}+c_{n-1}-c_{n-2}+\dots -c_3+c_3-c_2+c_2=\Delta_{n-1}+\Delta_{n-2}+\dots+\Delta_2+c_2$, then using the closed forms for $\Delta_{2^km}$ from \Cref{lem:firstDiffClosed} in the cases $m=1,m=3,m\geq 5$.
\end{proof}
This could be massaged further into an expression where the $O(n)$ terms in the sum are combined by common $\lfloor \log_2\frac{n-1}{2j+1}\rfloor$ terms to result in an (too complex) $O(\log n)$ sum, but the authors have not found an $O(1)$ closed form, which would be the desired closed form. We also note that the original recurrence $c_n=\lfloor\frac{n}{2}\rfloor\lceil\frac{n}{2}\rceil+c_{\lceil\frac{n}{2}\rceil}$ provides a way to compute the terms $c_n$ in $O(\log n)$ time already. 
\begin{corollary}\label{cor:recurrence_bounds}
    The recurrence $c_n$ obeys the bounds $\frac{n^2}{3}-1\leq c_n\leq \frac{n^2+n}{3}$ for all $n\geq 2$.
\end{corollary}
\begin{proof}
    This is immediate from an inductive proof using the recurrence.
\end{proof}

%% file: Sections/Complete/HCisH.tex
We now need to show that the herder binary search algorithm is optimal against a cat that follows known strategy $C$. Namely, we must show that any other herder strategy $H'$ may be replaced with $H$ without increasing the score.

Proving \Cref{thm:complete} is the target over the next few pages, but we will need a few prerequisite lemmas.

\begin{theorem}\label{thm:complete}
    For all $n\geq 2$, $\cat(K_n)=c_n$, where $c_n$ is defined by the recurrence $c_2=1,c_3=2,c_n=\lfloor\frac{n}{2}\rfloor\lceil\frac{n}{2}\rceil+c_{\lceil \frac{n}{2}\rceil}$ for $n>3$.
\end{theorem}

We henceforth will assume that $\cat(K_n)\neq c_n$ for some $n$. We will consider the minimal $n$ such that $\cat(K_n)\neq c_n$. For such $n$, if $\cat(K_n,H,C)=\cat(K_n,C)$, then we know that $H,C$ are optimal (by \cref{lrScore} and \cref{KnHCisC}), and $\cat(K_n)=c_n$ as $H$ always scores $c_n$. This is a contradiction, thus $\cat(K_n,H,C)>\cat(K_n,C)$, so that $H'$ must be some strategy that can isolate the cat faster against $C$. Let $H'$ be such a strategy. We will proceed to deduce various properties that $H'$'s critical graph cannot have until we have fully deduced that no such $H'$ can exist as there is no possible corresponding critical graph.

It will be convenient to assume that each $2$-edge-connected component is complete, we show the following lemma to make this assumption without loss of generality.

\begin{lemma}\label{lem:assumeCompComps}
    For any strategy $H'$ that is optimal against $C$, there exists a strategy $H''$ that is optimal against $C$ such that the critical graph $G$ has all $2$-edge-connected components complete.
\end{lemma}
\begin{proof}
    Let $H'$ be a herder strategy, and let $G$ be the first critical graph reached on play against $C$. Let $C_i$ be the maximal $2$-edge-connected subgraphs' vertices, as defined in \Cref{alg:general}. Let $A=\{uv\not\in G | u\in C_i,v\in C_j,i\neq j\}$ be the set of all edges that needed to be cut in order to reach the critical graph. Let $F=(E(K_n)\setminus E(G))\setminus A$ be the set of all edges that are cut within $2$-edge-connected components of the critical graph.

    \par Strategy $H'$ may now be described as cutting $A\cup F$, say ending with edge $e$ that drops the graph's edge connectivity to one.

    \par We introduce strategy $H''$ which will be to cut $A$, say ending with edge $e$, and then deleting $F'\subseteq F$. When $F'$ is not specified in the following analysis, it should be assumed to be $F'=\emptyset$. Assume that $F\neq \emptyset$ so that $H'\neq H''$ and we don't have a trivial argument.

    \par We must show that $\cat(K_n,C,H')\geq \cat(K_n,C,H'')$. We will do so by considering where the cat moves under strategy $H'$ after edge $e$ is cut.
    \begin{enumerate}
        \item When $n\leq 9$ and $\lambda_k=3$ or $\lambda=(6,1)$, the cat chooses a move according to \Cref{alg:nSmall}.
        \begin{enumerate}
            \item Assume that $\lambda_k=3$. These are not possible as $F\neq \emptyset$ implies that some cut was made within a $2$-edge-connected component, which is not possible if $\lambda_k=3$.
            \item Otherwise assume $\lambda=(6,1)$. In this case, the cat will voluntarily restrict their play to the largest $2$-edge-connected component (until a new critical $\lambda$ is reached for c). As such, they enter case 6 and we will resolve that separately.
        \end{enumerate}
        \item Suppose that $\lambda_k=1$. Note that $\lambda_k=1$ is not possible as $F\neq \emptyset$. 

        \item Suppose that $\lambda=(3(2^m),3(2^m),1)$ for some integer $m$. In this case, the cat moves to the center component (of the three), and after one more cut the cat plays in partition $(3(2^m),1)$, with $|F|$ extra cuts from $F$. In the parallel game with $H''$, the same game is played, but resulting in $K_{3(2^m)}$ with a pendant bridge, and no extra cuts. As play proceeds, we end up in case 6 (or 1(b)), with the sole condition that we need to check that the cat will always be on a vertex that is of degree at least 3, so that the next move is still legal. Noting that if all vertices are of degree at most 2, then $\lambda_k=1$ or $k=1$, and since the cat goes to the vertex of highest degree, we satisfy the needed condition.
    
        \item Suppose that $\lambda_k-\lambda_{k-1}\leq 1$, so that the cat chooses a move according to \Cref{alg:nearlyBin}.

        \begin{enumerate}
            \item  Suppose further that $\lambda_k=\lambda_{k-1}$. If $F$ contains cuts from both components $C_k,C_{k-1}$, then the cat moves to the one with fewer cuts (say $C_k$) and voluntarily restricts themselves to that component. This play by $H'$ scores no worse than the binary search herder $H$, who will use the saved cut from $C_{k-1}$ to cut the bridge and proceed to score the same result. In the case $F$ contains cuts from only one component, the cat freely moves to the other and uses the savings to count as cutting the bridge.
    
            \item Alternately assume that $\lambda_k=\lambda_{k-1}+1$. If $F$ contains any cuts from $C_k$, then the cat voluntarily restricts play to $C_{k-1}$, as if the bridge were cut. Otherwise, all cuts are presently in $C_{k-1}$, and the cat proceeds as normal, providing the cat against $H'$ at least as many cuts as against $H''$.
        \end{enumerate}
        \item Suppose that  $\lambda=(3(2^m),1)$. This mirrors the structure of 1(b), it will collapse to case 6, since the cat will never move to the bridge, and the herder will never cut it until the $\lambda$ changes.
    
        \item Otherwise, the cat will run the final else case in \Cref{alg:general}, and will move to a large connected subgraph. In particular, the only thing that the cat cares about is moving to a vertex with degree at least $2$ (to avoid edge cases of \Cref{alg:nearlyBin}). However, if the herder is approaching a critical graph of the form $(3(2^m),3(2^m),1)$, then the cat will move to the center component. Since the $3(2^m)$ components are $2$-edge-connected, all vertices have degree at least $2$. Combined with the bridges, both components must have at least one vertex of degree at least $3$. Thus if the center component is of size $1$, then the cat is not presently there, as the cat must be on a vertex of degree at least $3$ (as there must be at least $2$ options). This means that in $H',H''$, the cats will always be in the same $2$-edge-connected component, or otherwise $H''$ does not increase the score of $H'$.
    \end{enumerate}
    Thus regardless of how $H'$ forces $C$ to move, we have $\cat(K_n,C,H')\geq \cat(K_n,C,H'')$.
\end{proof}
Due to \Cref{lem:assumeCompComps}, we will henceforth always assume that the critical graph has all $2$-edge-connected components complete, thus enabling inductive arguments.

The following lemma describes a key idea that is used throughout the case arguments that follow.
\begin{lemma}\label{lem:standardCompleteAnalysis}
    If strategy $H'$ against $C$ on minimal counterexample $K_n$ leads to a critical graph with corresponding integer partition $\lambda$ of size $k$ such that $\displaystyle \sum_{1\leq i<j\leq k}\lambda_i\lambda_j-k+\frac{1}{3}\lambda_k^2\geq \frac{1}{3}n^2+\frac{7}{12}n+\frac{1}{4}$, then $H'$ scores at least as high as $H$.
\end{lemma}
\begin{proof}
    This follows through a straight-forward inequality chasing, and will be the foundation for many of the following arguments.
    \begin{align*}
        \sum_{1\leq i<j\leq k}\lambda_i\lambda_j-(k-1)+\cat(\lambda_k)
        &\label{eq:ind}\geq \sum_{1\leq i<j\leq k}\lambda_i\lambda_j-(k-1)+\frac{1}{3}\lambda_k^2-1\\
        &= \sum_{1\leq i<j\leq k}\lambda_i\lambda_j-k+\frac{1}{3}\lambda_k^2\\
        &\geq \frac{1}{3}n^2+\frac{7}{12}n+\frac{1}{4}\\
        &\geq \left\lceil\frac{n}{2}\right\rceil\left\lfloor \frac{n}{2}\right\rfloor+\frac{1}{3}\left\lceil\frac{n}{2}\right\rceil\left(\left\lceil\frac{n}{2}\right\rceil+1\right)\\
        &\geq \left\lceil\frac{n}{2}\right\rceil\left\lfloor \frac{n}{2}\right\rfloor+\cat\left(\left\lceil\frac{n}{2}\right\rceil\right)
    \end{align*}
    By our assumption of $n$ being a minimal counterexample, we get our first and last 
    inequalities from \Cref{thm:mainComplete} on $\lceil\frac{n}{2}\rceil$ and $\lambda_k$.
    
    Noting that the left hand side corresponds to a cat playing strategy $C'$ resulting in complete $K_{\lambda_k}$, as well as the fact that the right hand side corresponds exactly to the herder binary search strategy, this gives us a direct comparison tool for when the structure of $\lambda,n$ is highly restricted.
\end{proof}

Some small cases of our argument are computer-assisted, and code for \texttt{testPartition} is available in the appendix.
\begin{lemma}\label{KnSmallCasesCode}
    The algorithm \texttt{testPartition} correctly computes a lower bound on the cat number of the corresponding herder strategy.
\end{lemma}
\begin{proof}
    We note that the cost computed in \texttt{testPartition} is $\sum_{1\leq i<j\leq k}\lambda_i\lambda_j-(k-1)+c_{\lambda_k}$, with an adjustment made in each of the following cases.

    First, if $\lambda_k=1$, then we have a path, and know that the cat can proceed to score at least the corresponding path score, $\lceil\log_2n\rceil$ for even paths (as either center is ok), or $\lceil\log_2n\rceil-1$ for odd paths, as the cat may have started on the center of the path.

    If $\lambda=(3,1,\dots)$, then the cat may move to a vertex in $K_3$ non-incident to some pendant edge. After any cut, the cat may move to the center of a $P_3$ (either in $K_3$ or using the other two edges of $K_3$ not cut and the pendant edge) and proceed to score $2$ more, one more than $K_3$ would. This yields the necessary extra point.

    If $\lambda_2-\lambda_1\leq 1$, then the cat may score an extra point beyond the naive analysis by a) if $\lambda_2=\lambda_1$, after the next cut, moving to a complete $K_{\lambda_2}$ which must exist (as either the bridge was cut and the cat is already in one, or it was not and there is a path to the other one). Alternately, b) if $\lambda_2=\lambda_1+1$, then the cat makes $\lambda_1-3$ moves away from any edge incident to the bridge in $K_{\lambda_2}$. After the next cut in $K_{\lambda_2}$, the cat moves to the bridge (on the $\lambda_2$ side). After the next cut in $K_{\lambda_2}$, the cat moves to $K_{\lambda_1}$ adjacent to the bridge. The herder has at this point spend $\lambda_1-1$ cuts. If the herder cuts outside of $K_{\lambda_1}$ at this point, then the cat plays in $K_{\lambda_1}$ scoring $\lambda_1+K_{\lambda_1}\geq K_{\lambda_2}$. Otherwise, at some point in this process, the herder must have cut outside of $K_{\lambda_2}$ when the cat had access to $K_{\lambda_2}$, allowing the cat to play optimally in $K_{\lambda_2}$. 

    The final cases are when $\lambda=(6,1),(12,1),(24,1)$ or $\lambda=(3,3,1)$. In the case of $(3,3,1)$, the cat may move to the center and on any cut, follow the $(3,1)$ strategy for $2$ extra points. In the case of $(6,1)$, the herder will not cut the extra edge (lest they give the extra point needed), and the cat will not leave $K_6$ (lest they lose next turn). Thus both will play in $K_6$. Play will continue until a new partition is reached. When this occurs, this will be a new partition of $K_7$ when the $1$ is appended, and a prior analysis will apply (either $(5, 1, 1), (4, 1, 1, 1), (3, 3, 1), (3, 1, 1, 1, 1), (1, 1, 1, 1, 1, 1, 1)$). Special care should be given to the case when the new partition is all ones, as this may no longer be a path of length $7$, but rather $6$. However, this does not change $\lceil\log_2n\rceil$, and so this does not change the result. Similarly, if $\lambda=(12,1),(24,1)$, only one point is saved relative to binary search. As such, the herder will not cut the extra edge, and play will proceed to a new partition. Since we must end up at a partition of $13,25$ that contains a $1$, we can observe that the only one that seems to save a point is $(6,6,1), (12,12,1)$, respectively. But then the cat will play to the center component, and play in $K_6+e,K_{12}+e$ afterwards, thus gaining the extra point that was missed and they can proceed to play in the largest component.
    
    The additional cases are all discharged with a direct naive analysis, with the sole exception of $(3,3,3)$, which can be discharged by noting it contains a $(3,1)$ substructure which we know allows the cat to score an extra point.

    All told, \texttt{testPartiion} computes a lower bound for the given herder strategy's partitions, and verifying these are all at least $c_n$ shows that no strategy of these cases improves upon the binary herder search strategy $H$.
\end{proof}

\begin{lemma}\label{KnSmallCases}
   In each of the following cases for a critical graph's $\lambda$, the corresponding strategy $H'$ may be replaced by binary search herder $H$ without loss of performance. Namely, the corresponding strategy $H'$ captures the cat on $K_n$ in at least $c_n$ moves.

    \begin{enumerate}
        \item $n<9$. There are 65 different partitions with $2\leq n\leq 8, k\geq 2, 2\not\in\lambda$.
        \item $\lambda_{k-1}=1, k\in \{3,4,5,6,7,8,9,10,11,12\}, \lambda_k\leq m_k$, where $\{m_k\}_3^{12}=\{56,14,9,6,5,4,3,1,1,1\}$, respectively. Partitions are of form $(\lambda_k,1,\dots, 1)$.
        \item $\lambda_3=\lambda_k,\lambda_2=1, k\in \{3,4\}, \lambda_3<m_k$, where $\{m_k\}_3^4=\{57,35\}$, respectively. Partitions are of form $(\lambda_3,1,1),(\lambda_3,\lambda_3,1,1)$.
        \item $k=2$, $n\leq 27$
        \item $k=3, \lambda_k=\lambda_2, \lambda_1\in \{3,4,5,6\}, \lambda_2\leq m_k$, where $\{m_k\}_3^6=\{8,6,6,6\}$, respectively. Partitions are of form $(\lambda_2,\lambda_2,\lambda_1)$
        \item $k=4, \lambda_k=\lambda_2, \lambda_2\leq 4$. Partitions are of form $(\lambda_2,\lambda_2,\lambda_2,\lambda_1)$.
        \item $k=5, \lambda_k=\lambda_2, \lambda_1=1,\lambda_2=3,k=5$. Partition is $(3,3,3,3,1)$.
    \end{enumerate}
    
\end{lemma}
\begin{proof}
    The correctness of \texttt{testPartition} and the fact that the code in input 1 generates all relevant partitions provides a proof by code that these exceptional cases satisfy the lemma. 
\end{proof}
\begin{lemma}\label{Knl2}
    Let $H'$ be an optimal herder strategy on $K_n$, where $n$ is a minimal counterexample to \Cref{thm:mainComplete}. Let $(\lambda_i)$ be the partition corresponding to the critical graph in a game of $C$ as the cat, and $H'$ as the herder. If $k=2$, then $H'$ captures the $C$ on $K_n$ in no fewer than $c_n$ moves.
\end{lemma}
\begin{proof}
    Let $H'$ be a counterexample that is minimal with respect to $n$, and then $\lambda_2-\lambda_1$. Note then that we know that $\cat(K_{\lambda_i})=c_{\lambda_i}$ for $i\in \{1,2\}$ by minimality with respect to $n$. We consider cases based on $\lambda_2-\lambda_1$.
    \begin{enumerate}
        \item $\lambda_2=\lambda_1$. In this case, all we must show is that from here it is an optimal herder move to cut the bridge. If the herder doesn't, the cat strategy is to move to the component that did not get the cut, and play as if the bridge was cut. Thus the strategy $H'$ has an optimal continuation to play as $H$ would, and thus $H'$ does not isolate the cat quicker, a contradiction.
        \item $\lambda_2=\lambda_1+1$. We must show that the herder's strategy is to cut the bridge (say $v_1v_2$ with $v_2$ in the $K_{\lambda_2}$ component, and $v_1$ in the $K_{\lambda_1}$ component), and if they do not, then the cat can score at least as well as if they did cut $v_1v_2$ initially. Assume that the herder cuts $\lambda_1-d$ cuts in $K_{\lambda_2}$, then cuts somewhere else, say $e$, for $d\geq 1$. Then the cat may assume that rather than $e$, the herder's cut was on the bridge, and continue their optimal strategy of moving to the largest $2$-edge-connected component, as it is impossible for a critical graph to have been reached. In this case, the cat gets to act as if the herder had cut the bridge, without loss of score. 
        \par Assume alternately that the first $\lambda_1$ cuts the herder makes are in $K_{\lambda_2}$. Then the cat's strategy is for the first $\lambda_1-2$ cuts to always move to any vertex in $K_{\lambda_2}$ that is not $v_2$ (possible since $K_{\lambda_2}$ is $\lambda_1$-edge-connected, and only $\lambda_1-3$ cuts have been made. After the $(\lambda_1-2)$nd cut, the cat moves to $v_2$. The herder must cut their $(\lambda_1-1)$st cut in $K_{\lambda_2}$ again, leaving the cat available to move to  $K_{\lambda_1}$, say on $v_1$. If the herder's $\lambda_1$st cut is in $K_{\lambda_2}$ or the bridge, the cat continues play in $K_{\lambda_1}$ as if the bridge was cut, scoring $\lambda_1+\cat(\lambda_1)$ which we inductively know is $\lambda_1+c_{\lambda_1}\geq \Delta_{\lambda_1}+c_{\lambda_1}=c_{\lambda_2}$, scoring the cat the necessary score. Alternately, the herder's $\lambda_1$st cut is in $K_{\lambda_1}$, and there were $\lambda_1-1$ cuts in $K_{\lambda_2}$, so the cat may return there and play as if the bridge was cut, again attaining the desired score.
        \item
        \begin{enumerate}
            \item $\lambda_2 \geq \lambda_1+2$, $\lambda_1> \frac{1}{2}(\lambda_2-1)$. When $\lambda_1> \frac{1}{2}(\lambda_2-1)$, we have that $\lambda_2-\lambda_1-1< \frac{\lambda_2-1}{2}\leq \Delta_{\lambda_2-1}=c_{\lambda_2}-c_{\lambda_2-1}=\cat(K_{\lambda_2})-\cat(K_{\lambda_2-1})$ by induction and \Cref{lem:delNBounds}. Rearranging and adding $\lambda_1\lambda_2$ to both sides, we get that $\lambda_1\lambda_2-\lambda_1+\lambda_2-1+\cat(\lambda_2-1)< \lambda_2\lambda_1+\cat(\lambda_2)$. This can be factored as $(\lambda_2-1)(\lambda_1+1)+\cat(\lambda_2-1)\leq \lambda_2\lambda_1-1+\cat(\lambda_2)$. Note that the right hand corresponds to the score of the game under herder strategy $H$, and the left hand corresponds to the score of the game under an altered strategy $H'$ where the partition donates one vertex from the larger component to the smaller one, and then fully splits the components and recurses. Thus a slight alteration closer to the proposed optimal strategy does not lose performance, and this violates our choice of $H'$ as minimal with respect to $\lambda_2-\lambda_1$.
            
            \item Set $p=\lambda_2/n$. We observe the following chain of inequalities (note that the left hand side corresponds to the game played where the herder cuts $pn$ from $(1-p)n$ and leaves one bridge).
            \begin{align*}
                (pn)(1-p)n-1+c_{pn}
                &=n^2p(1-p)+c_{pn}-1\\
                &\geq n^2p(1-p)+\frac{(pn)^2}{3}-2\\
                &=n^2p(1-p)+\frac{p^2n^2}{3}-2\\
                &=n^2\left(p-p^2+\frac{p^2}{3}\right)-2\\
                &=n^2\left(p-\frac{2p^2}{3}\right)-2\\
                &=n^2\left(\frac{n-\lambda_1}{n}-\frac{2(\frac{n-\lambda_1}{n})^2}{3}\right)-2\\
                &=n^2\left(1-\frac{\lambda_1}{n}-\frac{2}{3n^2}(n^2-2n\lambda_1+\lambda_1^2)\right)-2\\
                &=n^2-\lambda_1n-\frac{2n^2}{3}+\frac{4n\lambda_1}{3}-\frac{2\lambda_1^2}{3}-2\\
                &=\frac{n^2}{3}+\frac{n\lambda_1}{3}-\frac{2\lambda_1^2}{3}-2\\
                &=\frac{4n^2+4n\lambda_1-8\lambda_1^2-24}{12}
            \end{align*}
            On the other hand, we look at the direct binary search strategy's bounds
            \begin{align*}
                \frac{4n^2+7n+3}{12}
                &=\frac{n^2+n}{4}+\frac{n^2+4n+3}{12}\\
                &=\frac{n+1}{2}(\frac{n}{2})+\frac{(\frac{n+1}{2})^2+\frac{n+1}{2}}{3}\\
                &\geq\lceil\frac{n}{2}\rceil\lfloor\frac{n}{2}\rfloor+\frac{\lceil\frac{n}{2}\rceil^2+\lceil\frac{n}{2}\rceil}{3}\\
                &\geq\lceil\frac{n}{2}\rceil\lfloor\frac{n}{2}\rfloor+c_{\lceil\frac{n}{2}\rceil}
            \end{align*}
            Multiplying both by 12, we now have a goal of showing 
            \[4n^2+4n\lambda_1-8\lambda_1^2-24\geq 4n^2+7n+3\]
            If we show this, we can combine the two inequalities, and show that a direct comparison between $H'$ and $H$ shows that $H'$ doesn't lower the score relative to $H$.
            Indeed, this follows from $4n\lambda_1-8\lambda_1^2\geq 7n+27$.

            We must show that $4\lambda_1(n-2\lambda_1)\geq 7n+27$. We note that $\lambda_1\leq \frac{1}{2}(\lambda_2-1)$ implies that $2\lambda_1\leq \lambda_2$, and thus $\lambda_2-\lambda_1\geq \frac{\lambda_1+\lambda_2}{3}=\frac{n}{3}$. Assume that $\lambda_1\geq 6$ and $n\geq 27$.
            \begin{align*}
                4\lambda_1(n-2\lambda_1)
                &=4\lambda_1(\lambda_2-\lambda_1)\\
                &\geq 24(\lambda_2-\lambda_1)\\
                &\geq 24(\frac{n}{3})\\
                &=8n
                &\geq 7n+27
            \end{align*}
            If $n<27$ and $\lambda_1\geq 6$, we note that this is covered by \Cref{KnSmallCases}
            Lastly, each of $\lambda_1\in \{1,3,4,5\}$ we can check by hand. For $\lambda_1\in \{3,4,5\}$, we get the target equation $4\lambda_1(n-2\lambda_1)\geq 7n+27$ true for $\lambda_1=3, n\geq 20$, $\lambda_1=4, n\geq 18$, $\lambda_1=5, n\geq 18$. Again, these exceptional small cases are all covered by \Cref{KnSmallCases}. Finally, if $\lambda=(n-1,1)$, then the herder has cut $n-2$ cuts to end up with $K_{n-1}$ with a tail. We know that if the cat plays in $K_{n-1}$ only, they will score $n-2+c_{n-1}$. Suppose for contradiction that this scores better than binary search, so that $n-2+c_{n-1}<c_n$. Then $n-2<\Delta_{n-1}$. Set $n'=n-1$, so that $n'-1<\Delta_{n'}$. Noting that $\Delta_{n'}\leq n'$ from \Cref{lem:delNBounds}, we obtain that $\Delta_{n'}=n'$, and thus that \Cref{lem:firstDiffClosed} tells us that $n'=3(2^a)$ for some integer $a$, as no other type of integer can have $\Delta_{n'}=n'$. But that means that the herder has cut to the conditions of \Cref{Kn1weird}, and we actually know that the cat can achieve one more point than otherwise expected. Revisiting the analysis, this means that the herder does not lose anything by cutting the edge pendant, and so without loss of generality the equation now becomes $n-1+c_{n-1}<c_n$, now implying that $n'<\Delta_{n'}$, a contradiction.
        \end{enumerate}
    \end{enumerate}
    Thus we have seen that if $k=2$, regardless of how $H'$ works, $C$ will avoid capture in $K_n$ for at least $c_n$ moves.
\end{proof}
\begin{lemma}\label{KnkBiglk11}
    Let $H'$ be an optimal herder strategy on $K_n$, where $n$ is a minimal counterexample to \Cref{thm:mainComplete}. Let $(\lambda_i)$ be the partition corresponding to the critical graph in a game of $C$ as the cat, and $H'$ as the herder. If $k\geq 3,$ and $\lambda_{k-1}=1$, then $H'$ captures the cat playing strategy $C$ on $K_n$ in no fewer than $c_n$ moves.
\end{lemma}
\begin{proof}
    \par Let $H'$ be a counterexample that results in minimal $n$, then, of those, with minimal $k$.
    \begin{enumerate}
        \item  $k=3,\lambda_{k-1}=1$: This must be a partition of the form $(n-2,1,1)$. We chase inequalities. Since $n\geq 6$, so that $\lfloor\frac{n}{2}\rfloor\geq 3$, and thus $n-4\geq \lceil\frac{n}{2}\rceil-1\geq \Delta_{\lceil\frac{n}{2}\rceil-1}$ by \Cref{lem:delNBounds}. Expanding $n=\lceil\frac{n}{2}\rceil+\lfloor\frac{n}{2}\rfloor$ and adding $\lceil\frac{n}{2}\rceil\lfloor\frac{n}{2}\rfloor-\lceil\frac{n}{2}\rceil\lfloor\frac{n}{2}\rfloor$ as a $0$ we may find that $2n-5+\left(\lceil\frac{n}{2}\rceil\lfloor\frac{n}{2}\rfloor-\lceil\frac{n}{2}\rceil-\lfloor\frac{n}{2}\rfloor+1-\lceil\frac{n}{2}\rceil\lfloor\frac{n}{2}\rfloor\right)\geq \Delta_{\lceil\frac{n}{2}\rceil-1}$. Factoring, we may obtain that $2n-5+\left(\lceil\frac{n-2}{2}\rceil\lfloor\frac{n-2}{2}\rfloor-\lceil\frac{n}{2}\rceil\lfloor\frac{n}{2}\rfloor\right)\geq \Delta_{\lceil\frac{n}{2}\rceil-1}$. Finally rearranging this using $\Delta_{\lceil\frac{n}{2}\rceil-1}= c_{\left\lceil\frac{n}{2}\right\rceil}-c_{\left\lceil\frac{n-2}{2}\right\rceil}=\cat(\lceil\frac{n}{2}\rceil)-\cat(\lceil\frac{n-2}{2}\rceil)$, we find that $2n-5+\lceil\frac{n-2}{2}\rceil\lfloor\frac{n-2}{2}\rfloor+\cat(\lceil\frac{n-2}{2}\rceil)\geq \lceil\frac{n}{2}\rceil\lfloor\frac{n}{2}\rfloor+\cat(\lceil\frac{n}{2}\rceil)$. This describes our direct comparison, since the right hand side is the normal binary search strategy, while the left hand side represents the cost of the $(n-2,1,1)$ strategy that $H'$ is using. Indeed, there are $1(n-2)+1(n-1)+1(1)-2=2n-5$ cuts to reach the critical graph, after which the cat will play in a complete graph on $n-2<n$ vertices, which must inductively follow the known optimal strategy.
        \item $k>3,\lambda_{k-1}=1$: We will show that this strategy is strictly worse than the binary search strategy by direct comparison. We begin by showing that either we have small cases evaluated in \Cref{KnSmallCases}, or $2k^2+4k\lambda_k-29k-11\lambda_k+12\geq 0$. If $k=4$, then $\lambda_k\leq 14$ is in the small cases. Similarly, for $k\in (5,6,7,8,9,10,11,12,13,14)$, then $\lambda_k\leq (9,6,5,4,3,1,1,1)$, respectively, are all in the small cases. Otherwise, for $k<15$, directly substituting in the corresponding values of $k$ and using the bounds on $\lambda_k$ (noting $\lambda_k$ is an integer and is not $2$) give the claimed inequality, $2k^2+4k\lambda_k-29k-11\lambda_k+12\geq 0$. Finally, for $k\geq 15$, it is easy to rearrange the inequality as $k(2k-29)+\lambda_k(4k-11)+12\geq 0$, which is transparently true for $k\geq 15$. This inequality is sufficient to verify the condition of \cref{lem:standardCompleteAnalysis} (noting $\lambda=(\lambda_k,1,\dots, 1)$ implies $n=\lambda_k+k-1$ and $\displaystyle\sum_{1\leq i<j\leq k}\lambda_i\lambda_j=\lambda_k(k-1)+\frac{k(k-1)}{2}$), and thus $H'$ is not a counterexample.
    \end{enumerate}
    Thus in this case the herder can again not increase performance against $C$ relative to binary search.
\end{proof}
Note that we could have put $k=3$ in with the second analysis easily, it just adds the case of $k=3,\lambda_k\leq 56$ to the small cases, which may be viewed as a sufficient nuissance to want a different way to handle it.

The case of partitions of the form $(3(2^a),3(2^a),1)$ will require special consideration. We will need the following lemma which presupposes $\cat(K_n)=c_n$ for small $n$. We will ensure that this condition is met whenever we use \Cref{Kn1weird}.
\begin{lemma}\label{Kn1weird}
    For all $a\in\mathbb{N}$, if $\cat(K_{n})=c_{n}$ for $2\leq n\leq 3(2^a)+1$ and $G$ is $K_{3(2^a)}$ with one extra pendant edge, then $\cat(G)=\cat(K_{3(2^a)})+1$.
\end{lemma}
\begin{proof}
    It is clear that $\cat(G)\leq \cat(K_{3(2^a)})+1$ since the herder may cut the pendant edge, then play optimally in $K_{3(2^a)}$.

    Suppose that $\cat(G)<\cat(K_{3(2^a)})+1$. Then consider a game on $K_{3(2^a)+1}$. A possible partition is $(3(2^a),1)$. Such a position requires $3(2^a)-1$ cuts to reach, and the cat will optimally proceed to score $\cat(G)$. As such, this partition overall scores $3(2^a)-1+\cat(G)\geq \cat(K_{3(2^a)+1})$. We thus have that 
    \begin{align*}
        \cat(K_{3(2^a)})+3(2^a)
        &=\cat(K_{3(2^a)})+\Delta_{3(2^a)}&(\Cref{lem:firstDiffClosed})\\
        &=\cat(K_{3(2^a)})+c_{3(2^a)+1}-c_{3(2^a)}\\
        &=\cat(K_{3(2^a)})+\cat(K_{3(2^a)+1})-\cat(K_{3(2^a)})\\
        &=\cat(K_{3(2^a)+1})\\
        &\leq 3(2^a)-1+\cat(G)\\
        &<3(2^a)-1+\cat(K_{3(2^a)})+1\\
        &=\cat(K_{3(2^a)})+3(2^a)
    \end{align*}
    We have thus deduced a contradiction, and so we must have that $\cat(G)=\cat(K_{3(2^a)})+1$.
\end{proof}
Now we may resume with our usual minimal counter-example lemmas. We next assume that there is only one small component, and all other components are the same size.
\begin{lemma}\label{KnkBiglk1Bigl2Islk}
    Let $H'$ be an optimal herder strategy on $K_n$, where $n$ is a minimal counterexample to \Cref{thm:mainComplete}. Let $(\lambda_i)$ be the partition corresponding to the critical graph in a game of $C$ as the cat, and $H'$ as the herder. If $k>2, \lambda_{k-1}\geq 3$, and $\lambda_2=\lambda_k$, then $H'$ captures the cat playing $C$ on $K_n$ in no fewer than $c_n$ moves.
\end{lemma}
\begin{proof}
    \par By assumption, the partition is of the form $(\lambda_2,\dots, \lambda_2, \lambda_1)$. 
    First, the two cases we handle differently.
    \begin{enumerate}
        \item  $k=3, \lambda_1=1, \lambda_3\neq 3(2^a)$, so that $\lambda=(\lambda_3,\lambda_3,1)$. We tally the final score of $H'$. The initial cost is $\lambda_3^2+2\lambda_3-2$. The cat will use the strategy where they go to the middle component, gaining one point before voluntarily restricting their play to $K_{\lambda_3}$, and then will recursively score $\cat(K_{\lambda_3})$. Collectively, the final score will be $\lambda_3^2+2\lambda_3-1+\cat(K_{\lambda_3})$. We note that the binary search strategy has an initial cost of $\lambda_3(\lambda_3+1)$ before recursively scoring $\cat(K_{\lambda_3+1})$. Collectively, the final score will be $\lambda_3^2+\lambda_3+\cat(K_{\lambda_3+1})$. We now note that since $\lambda_3\neq 3(2^a)$, we know that $\Delta_{\lambda_3}<\lambda_3$. We compute
            \begin{align*}
                \lambda_3^2+2\lambda_3-1+\cat(K_{\lambda_3})
                & \geq \lambda_3^2 + 2\lambda_3 + \cat(K_{\lambda_3}) + \Delta_{\lambda_3} - \lambda_3 &(\lambda_3-1\geq \Delta_{\lambda_3})\\
                &= \lambda_3^2+\lambda_3+\cat(K_{\lambda_3})+\Delta_{\lambda_3}\\
                &= \lambda_3^2+\lambda_3+\cat(K_{\lambda_3+1})
            \end{align*}
            This shows that the strategy $H'$ does not improve the performance over $H$.
            
            \item $k=3, \lambda_1=1, \lambda_3=3(2^a),a\in\mathbb{N}$, so that $\lambda=(3(2^a),3(2^a),1)$. We tally the final score of $H'$. The initial cost is $9(2^{2a})+3(2^{a+1})-2$. The cat again moves to the middle component, thereby scoring $1$ more, and then playing in $K_{3(2^a)}$ with a pendant leaf. However, since $n=3(2^{a+1})+1\geq 3(2^a)+1$ is the minimal $n$ violating $\cat(K_n)=c_n$, this means that all $\cat(K_{n'})=c_{n'}$ for $n'\leq 3(2^a)+1$, and \cref{Kn1weird} is satisfied. Thus the cat will recursively score $\cat(K_{3(2^a)})+1$. Overall, the cat has scored $9(2^{2a})+3(2^{a+1})-2+1+\cat(K_{3(2^a)})+1$. The binary search strategy will score $3(2^a)(3(2^a)+1)+\cat(K_{3(2^a)+1})$. We may now check
            \begin{align*}
                9(2^{2a})+3(2^{a+1})-2+1+\cat(K_{3(2^a)})+1
                &=9(2^{2a})+3(2^{a+1})+\cat(K_{3(2^a)})\\
                &=9(2^{2a})+3(2^a)+\cat(K_{3(2^a)})+3(2^a)\\
                &=3(2^a)(3(2^a)+1)+\cat(K_{3(2^a)})+\Delta_{3(2^a)}\\
                &=3(2^a)(3(2^a)+1)+\cat(K_{3(2^a)+1})
            \end{align*}
            Thus in this particular case, the alternate strategy $H'$ scores \emph{exactly} as well as binary search. This is another optimal strategy for $n$ of the form $n=3(2^a)+1,a\geq 1$, but does not improve performance over binary search.
    \end{enumerate}
    
    In the remaining cases 3-7(b), we will aim to show the following inequality is true.
        $$2k^2\lambda_2^2+4k\lambda_1\lambda_2-10k\lambda_2^2-4\lambda_1^2-7k\lambda_2-4\lambda_1\lambda_2+12\lambda_2^2-12k-7\lambda_1+7\lambda_2+9\geq 0$$
    \begin{enumerate}
        \setcounter{enumi}{2}

        \item  $k=3, \lambda_1>1$. For $3\leq \lambda_1\leq 7$, it is easy to substitute both $k=3,\lambda_1$ in each possible case and get a linear equation in terms of $\lambda_2$. As such, the inequality is true when $\lambda_1\in (3,4,5,6,7), \lambda_2\geq (9,7,7,7,7)$, respectively. The remaining small cases are checked in \Cref{KnSmallCases}, or in the case of $\lambda_1=7$, we get for free that $\lambda_2\geq \lambda_1=7$. Henceforth assume that $\lambda_1\geq 8$. With the assumption that $\lambda_2\geq \lambda_1\geq 8,k=3$, it is not hard to verify the desired inequality in this case.
        \item $\lambda_2=\lambda_k, k=4$. In the case that $\lambda_2\leq 4$, we defer this to \Cref{KnSmallCases}. We assume that $\lambda_2\geq 5$. In the cases that $\lambda_1\in \{1,3\}$, we may substitute that into the desired inequality to get a quadratic in $\lambda_2$ which we may note is positive for $\lambda_2\geq 5$. We further assume without loss of generality that $\lambda_1\geq 4$. Then we may use $\lambda_2\geq 5,\lambda_2\geq \lambda_1\geq 3,k=4$ to obtain the desired inequality.
        \item $\lambda_2=\lambda_k, k=5$. First, if we substitute $\lambda_1=1$, then we find that we get a quadratic in $\lambda_2$ which is positive on $\lambda_2\geq 4$. We have $\lambda_2\geq 3$, but the remaining case ($\lambda_1=1,\lambda_2=3,k=5$) we leave for the small cases lemma. Assume $\lambda_1\geq 3$ then. It is not hard to verify that $\lambda_2\geq \lambda_1\geq 3,k=5$ implies the desired inequality.
    
        \item $\lambda_2=\lambda_k, k=6$. Knowing $\lambda_2\geq \lambda_1\geq 1,\lambda_2\geq 3,k=6$ is enough to derive the inequality again.

        \item
        \begin{enumerate}
            \item $\lambda_2=\lambda_k, k>=7, \lambda_2\geq 12$. This is the most general case, and we verify this by computing the following.
            \begin{align*}
                0&\leq (\lambda_2-12)k+4\lambda_2^2+9&(\lambda_2\geq 12)\\
                &\leq (8\lambda_1\lambda_2-7\lambda_2-12)k+4\lambda_2^2+9&(1\leq \lambda_1)\\
                &\leq (4\lambda_1\lambda_2+4\lambda_2^2-7\lambda_2-12)k+4\lambda_2^2+9 & (\lambda_1\leq \lambda_2)\\
                &\leq (2k\lambda_2^2+4\lambda_1\lambda_2-10\lambda_2^2-7\lambda_2-12)k+4\lambda_2^2+9&(7\leq k)\\
                &=2k^2\lambda_2^2+4k\lambda_1\lambda_2-10k\lambda_2^2-7k\lambda_2+4\lambda_2^2-12k+9&\\
                &\leq 2k^2\lambda_2^2+4k\lambda_1\lambda_2-10k\lambda_2^2-7k\lambda_2-4\lambda_1\lambda_2+8\lambda_2^2-12k+9 &(\lambda_1\leq \lambda_2)\\
                &\leq 2k^2\lambda_2^2+4k\lambda_1\lambda_2-10k\lambda_2^2-7k\lambda_2-4\lambda_1\lambda_2+8\lambda_2^2-12k-7\lambda_1+7\lambda_2+9 &(\lambda_1\leq \lambda_2)
            \end{align*}
            We have shown the inequality holds yet again.
    
            \item $\lambda_2=\lambda_k, k\geq 7, 3\leq \lambda_2\leq 11$. For each such $\lambda_2$, it is easy to verify that since $1\leq \lambda_1\leq \lambda_2$, there are only finitely many pairs $(\lambda_1,\lambda_2)$. For each of these, it is easy to substitute in the values to the desired inequality to get a quadratic in $k$, all of which are positive for $k\geq 5$, and we have $k\geq 7$, so again the desired inequality holds.
        \end{enumerate}
    \end{enumerate}
    Suffice to say, in cases 3-7(b), the claimed inequality holds. What does this inequality gain us? As usual, dividing by 12 and rearranging yields (*) in the following inequality chain.
    \begin{align*}
        1+\frac{(k-1)(k-2)}{2}\lambda_2^2&+(k-1)\lambda_1\lambda_2-(k-1)+\cat(K_{\lambda_2})\\
        &\geq 1+\frac{(k-1)(k-2)}{2}\lambda_2^2+(k-1)\lambda_1\lambda_2-k+\frac{\lambda_2}{3}\\
        &\geq \frac{1}{3}k^2\lambda_2^2+\frac{2}{3}k\lambda_1\lambda_2-\frac{2}{3}k\lambda_2^2+\frac{1}{3}\lambda_1^2+\frac{7}{12}k\lambda_2-\frac{2}{3}\lambda_1\lambda_2+\frac{1}{3}\lambda_2^2+\frac{7}{12}\lambda_1-\frac{7}{12}\lambda_2+\frac{1}{4}&(*)\\
        &\geq \frac{1}{12}(n+1)^2+\frac{1}{4}(n+1)n+\frac{1}{6}(n+1)\\
        &\geq \frac{1}{3}\left(\left\lceil\frac{n}{2}\right\rceil\right)\left(\left\lceil\frac{n}{2}\right\rceil+1\right)+\left\lceil\frac{n}{2}\right\rceil\left\lfloor\frac{n}{2}\right\rfloor\\
        &\geq \cat\left(\left\lceil\frac{n}{2}\right\rceil\right)+\left\lceil\frac{n}{2}\right\rceil\left\lfloor\frac{n}{2}\right\rfloor
    \end{align*}
    Noting that the left hand side corresponds to the cat playing against $H'$, counting an extra point for the cat employing step 4 of their strategy for an extra point (since $\lambda_2=\lambda_k$), while the right hand side corresponds to the binary search strategy. Thus the proposed strategy $H'$ fails to perform better than binary search in a direct comparison, which was what we had to show.
\end{proof}
Continuing with minimal counterexample possibilities, we now consider when we have multiple large and small components. 
\begin{lemma}\label{KnkBiglk1Bigl21}
    Let $H'$ be an optimal herder strategy on $K_n$, where $n$ is a minimal counterexample to \Cref{thm:mainComplete}. Let $(\lambda_i)$ be the partition corresponding to the critical graph in a game of $C$ as the cat, and $H'$ as the herder. If $k>2, \lambda_{k-1}\geq 3$, and $1=\lambda_2$, then $H'$ captures the cat playing stategy $C$ on $K_n$ in no fewer than $c_n$ moves.
\end{lemma}
\begin{proof}
    We have three cases, $\lambda_3=\lambda_k, 1<\lambda_3<\lambda_k$, and $\lambda_3=1$. The first case requires special analysis, and in the second two cases, we select $H'$ to be minimal with respect to $n$, then $k$. 

    First, the case with special considerations.
    \begin{enumerate}
        \item If $\lambda_3=\lambda_k$, then our partition is of the form $(\lambda_3,\dots, \lambda_3,1,1)$. We aim to do a direct comparison between this $\lambda$ and the alleged optimal strategy. We want to show the following inequality.

        \[2k^2\lambda_3^2-14k\lambda_3^2+k\lambda_3+24\lambda_3^2-12k-2\lambda_3-21\geq 0\]
    
        This inequality can be rearranged as $2(k-4)(k-3)\lambda_3^2+(k-2)\lambda_3-12k-21\geq 0$. Noting that for $k>4$, $2(k-4)(k-3)\geq 2$, so it suffices to show that $0\leq 2\lambda_3^2+(k-2)\lambda_3-12k-21$. However, this can be rearranged as $k(\lambda_3-12)+2\lambda_3^2-2\lambda_3-21$. Solving the linear $\lambda_3-12\geq 0$ and quadratic $2\lambda_3^2-2\lambda_3-21$ in $\lambda_3$ we find that we need $\lambda_3\geq 12$ and $\lambda_3\geq 4$. The cases where $\lambda_3\in \{1,3,4,5,6,7,8,9,10,11\}$ result in an inequality of quadratics in $k$ which all are positive for $k\geq 3$. In the cases when $k=3,4$, our desired equations become $\lambda_3-57\geq 0, 2\lambda_3-69\geq 0$, and we can note that the cases of $(\lambda_3,1,1)$ with $\lambda_3<57$ is handled in \Cref{KnSmallCases}, and $(\lambda_3,\lambda_3,1,1)$ with $\lambda_3<35$ is also handled in \Cref{KnSmallCases}. In short, we will now take this inequality as given. This inequality is sufficient to verify the condition in \cref{lem:standardCompleteAnalysis}, and thus shows that this strategy does not improve the performance.
    \end{enumerate}
    In the remaining cases, we will alter the strategy to a new $\lambda'$ which saves on the initial cost, but recursively will score at most one more, thus net not losing any performance from the herder's perspective. However, the altered strategy will violate $H'$s minimality with respect to $k$.
    \begin{enumerate}
        \setcounter{enumi}{1}
        \item If $1<\lambda_3<\lambda_k$, we have a current partition of the form $\lambda=(\lambda_k,\dots, \lambda_3,1,1)$. We define our alternate $\lambda'=(\lambda_k,\dots, \lambda_3+1,1)$ with $k-1$ components, rather than $k$. This saves the $\lambda_3-1$ cuts that would be spent to separate $\lambda_3,\lambda_1$ normally, which means again our initial cost is strictly less. Since $k\geq 4$ and $\lambda_3\neq 2$, there is no cat strategy in $C$ that may increase the recursive score, so it suffices to only check initial cost.

        \item If $\lambda_3=1$, then either $k=3$ and $\lambda=(1,1,1)$, which is \Cref{KnSmallCases}, or we have a partition of the form $(\lambda_k,\dots, \lambda_4,1,1,1)$. We alter this to the partition $\lambda'=(\lambda_k,\dots, \lambda_4,3)$. This clearly costs one cut less than $\lambda$ as the final triangle does not have an edge cut, and again the initial cost is strictly less. Here, the modified partition has $k'=2$ if and only if $\lambda_k=\lambda_4$, which is handled by \Cref{KnSmallCases} when $\lambda_4\in \{3,4\}$, and otherwise the cat will play in the $\lambda_4$ component without gaining any extra score. For $k'\geq 3$, though, the only special case handling is when $\lambda=(3(2^m),3(2^m),1)$, which only happens if the previous $\lambda=(3,1,1,1,1)$. In this case, we know that the cat can manage to score at most one more point from \Cref{Kn1weird} and the cat's strategy of playing on $(3(2^m),1)$ partitions. Thus this strategy scores no worse as a net effect, but it has $k-2$ components, and thus violates minimality with respect to cat number, then $k$.
    \end{enumerate}
    Thus we see that there can be no $H'$ of this type that captures cat $C$ faster than binary search.
\end{proof}
Next, we consider when the critical graph of a minimal counterexample has a partition that allows donating a vertex from the smalllest component to the next-smallest component.
\begin{lemma}\label{KnkBiglk1Bigl2Mid}
    Let $H'$ be an optimal herder strategy on $K_n$, where $n$ is a minimal counterexample to \Cref{thm:mainComplete}. Let $(\lambda_i)$ be the partition corresponding to the critical graph in a game of $C$ as the cat, and $H'$ as the herder. If $k>2, \lambda_{k-1}\geq 3$, and $1<\lambda_2<\lambda_k$, then $H'$ captures the cat playing strategy $C$ on $K_n$ in no fewer than $c_n$ moves.
\end{lemma}
\begin{proof}
    \par If such a strategy that scores better than $H$ meets this criteria, we select it such that the critical graph $G$ has a minimal $n$, and then minimal $k$, and then minimal $\lambda_1$.
    \par In this case, create a new $\lambda'$ by copying $\lambda$, but instead change it so that $\lambda_2'=\lambda_2+1$ and $\lambda_1'=\lambda_1-1$, then reindex if necessary (if $\lambda_1=1$ we removed a component or $\lambda_1=3$ results in breaking a triangle into $3$ isolated vertices, requiring re-indexing).
    \par The alternate herder's strategy is to cut the cat to a path of components with sizes $\lambda_k,\lambda_{k-1},\dots, \lambda_3,\lambda_1,\lambda_2$, where all bridge edges meet at the same vertex of each component. Note that $\lambda,\lambda'$ have initial costs of, respectively,
    $$\sum_{1\leq i<j\leq k}\lambda_i\lambda_j-(k-1), \sum_{1\leq i<j\leq |\lambda'|}\lambda_i'\lambda_j'-(|\lambda'|-1)$$
    If $\lambda_1\geq 4$, then $k=|\lambda'|$, and so we need only compare the $\lambda_i\lambda_j$ terms. We note that $\lambda_2-\lambda_1\geq 0$, and so $0>\lambda_1-\lambda_2-1$. Adding $\lambda_1\lambda_2$ to both sides we get that $\lambda_1\lambda_2>(\lambda_1-1)(\lambda_2+1)=\lambda_1'\lambda_2'$. Also noting that $\lambda_1+\lambda_2=\lambda_1'+\lambda_2'$, we can verify that 
    \begin{align*}
        \sum_{1\leq i<j\leq k}\lambda_i\lambda_j
        &=\sum_{3\leq i<j\leq k}\lambda_i\lambda_j+(\lambda_1+\lambda_2)\sum_{i=3}^{k}\lambda_i+\lambda_1\lambda_2\\
        &<\sum_{3\leq i<j\leq k}\lambda_i'\lambda_j'+(\lambda_1'+\lambda_2')\sum_{i=3}^{|\lambda'|}\lambda_i'+\lambda_1'\lambda_2'\\
        &=\sum_{1\leq i<j\leq |\lambda'|}\lambda_i'\lambda_j'
    \end{align*}
    Everything else being equal, this means that the initial cost of $\lambda'$ is strictly less than that of $\lambda$.
    \par We must still verify cases $\lambda_1=1$ and $\lambda_1=3$. Suppose $\lambda_1=3$, so that $\lambda'=(\lambda_k,\dots, \lambda_2+1,1,1)$ is now a partition of $k+1$ parts, since $K_2$ is not $2$-edge-connected. For simplicity, we will say that $\lambda_1'=1,\lambda_0'=1$ as our reindexing. Now we note that since $\lambda_2\geq \lambda_1=3$, we know that $3\lambda_2\geq 2(\lambda_2+1)+1=2\lambda_2'+1$, and thus $\lambda_1\lambda_2=3\lambda_2\geq 2\lambda_2'+1=(\lambda_0'+\lambda_1')\lambda_2'+\lambda_0'\lambda_1'$. We can now derive that 
    \begin{align*}
        \sum_{1\leq i<j\leq k}\lambda_i\lambda_j-(k-1)
        &=\sum_{3\leq i<j\leq k}\lambda_i\lambda_j+(\lambda_1+\lambda_2)\sum_{i=3}^{k}\lambda_i+\left(\lambda_1\lambda_2-(k-1)\right)\\
        &\geq \sum_{3\leq i<j\leq k}\lambda_i'\lambda_j'+(\lambda_0'+\lambda_1'+\lambda_2')\sum_{i=3}^{k}\lambda_i'+\left(\lambda_0'\lambda_1'+\lambda_0'\lambda_2'+\lambda_1'\lambda_2'-(k-1)\right)\\
        &=\sum_{0\leq i<j\leq k}\lambda_i'\lambda_j'-(k-1)\\
        &>\sum_{0\leq i<j\leq k}\lambda_i'\lambda_j'-(|\lambda'|-1)
    \end{align*}
    Again, this means that the initial cost of $\lambda'$ is strictly less than that of $\lambda$.
    Finally, we verify $\lambda_1=1$.
    
    Since $\lambda_2>1$ and corresponds to a $2$-edge-connected component, we have $\lambda_2\geq 3,\lambda_1=1$. In this case, we consider that $\lambda'=(\lambda_k,\dots, \lambda_2+1)$. Noting that $|\lambda'|=k-1$ and $\lambda_2\geq 3$, we know that $\lambda_1\lambda_2-(k-1)= \lambda_2-|\lambda'|> -(|\lambda'|-1)$. Using this we may derive
    \begin{align*}
        \sum_{1\leq i<j\leq k}\lambda_i\lambda_j-(k-1)
        &=\sum_{3\leq i<j\leq k}\lambda_i\lambda_j+(\lambda_1+\lambda_2)\sum_{i=3}^{k}\lambda_i+\left(\lambda_1\lambda_2-(k-1)\right)\\
        &> \sum_{3\leq i<j\leq k}\lambda_i'\lambda_j'+(\lambda_2')\sum_{i=3}^{k}\lambda_i'-(|\lambda'|-1)\\
        &=\sum_{2\leq i<j\leq k}\lambda_i'\lambda_j'-(|\lambda'|-1)
    \end{align*}

    Thus, yet again, this alternate $\lambda$ has a strictly lower setup score.

    All that remains to be shown is that in any of these cases, the recursive score remains the same. Note that since $\lambda_k>\lambda_2$, and in all cases we have increased $\lambda_2$ by $1$, we have that $\lambda_k'\geq \lambda_2'$. Note that the only way that the cat does not voluntarily restrict themselves to the largest component in the resulting game is when $\lambda'=(3(2^m),\ 3(2^m),1),\ \lambda'=(3(2^m),1)$, or $\lambda_k'-\lambda_{k-1}'\leq 1$. If either of the first two, then $\lambda=(3(2^m)$ or $3(2^m)-1,1,1),\ \lambda=(3(2^m)-1,1,1)$, respectively. These cases from \Cref{Kn1weird} score at most one more, thus giving at worst exactly the same score as a binary search herder. However, these $\lambda'$ have fewer components than $\lambda$, violating the minimality of $k$. Finally, if $\lambda_k'-\lambda_{k-1}'\leq 1$, then the cat plays according to \Cref{alg:nearlyBin}. This may happen because $k=2$ and the original partition can be handled by \Cref{Knl2}, because $k=3$ and original partition is $(\lambda_2+c,\lambda_2,\lambda_1)$ for $1\leq c\leq 2$ (using $\lambda_2<\lambda_k$), or because $k\geq 4$ (in which case the recursive game plays on $\lambda_k,\lambda_{k-1}$ while ignoring the rest of the graph). Without loss of generality, then, $k=3,\lambda=(\lambda_2+c,\lambda_2,\lambda_1)$ for $1\leq c\leq 2$. Suppose $c=1$. Then $\lambda_3'=\lambda_2'=\lambda_3$. The cat will play as if the herder cuts the bridge between the two components, and will score exactly as well in $\lambda'$ as in $\lambda$. Alternately, $c=2$ and in $\lambda$ the cat did not use \Cref{alg:nearlyBin}, but they will in $\lambda'$. However, following \Cref{alg:nearlyBin} scores results in a score where the cat does not evade capture relative to the case where the cat voluntarily restricts themselves to $K_{\lambda_3}$ after a bridge cut. Thus in either case, so nothing is gained for the cat.
\end{proof}
We are finally ready to prove \Cref{thm:mainComplete}.

\comCatThm*

\begin{proof}
    The closed form of $c_n$ and asymptotic bounds on $c_n$ have already been shown in \Cref{thm:recurrence_closed} and \Cref{cor:recurrence_bounds}. All that is left is the full proof that $\cat(K_n)=c_n$.
    
    It is clear from the herder strategy of binary search that $\cat(K_n)\leq c_n$. For sake of contradiction, suppose for some $n$, $\cat(K_n)<c_n$. Then if $\cat(K_n,C,H)=\cat(K_n,C)$, we would have a contradiction as $\cat(K_n)=c_n$ in that case. Thus $\cat(K_n,C,H)>\cat(K_n,C)$ so that there is a herder strategy $H'$ that can isolate the cat playing $C$ faster than binary search. For any $H'$, let $G$ be the first graph that is reached that is not $2$-edge-connected, denote this graph the critical graph. Let $\{C_i\}$ be the set of all maximal $2$-edge-connected subgraphs of $G$. Let $|V(C_i)|=\lambda_i$, and re-index as necessary to make $\lambda_k\geq \lambda_{k-1}\geq \dots\geq \lambda_1$.

    \par If $n<9$, then \cref{KnSmallCases} tells us that $H'$ does not improve the performance for the herder over binary search when playing against the cat with strategy $C$. Assume $n\geq 9$.

    \par If $k=2$, then \cref{Knl2} tells us that $H'$ does not improve the performance for the herder over binary search when playing against the cat with strategy $C$. Assume $k\geq 3$.

    \par If $\lambda_{k-1}=1$, then \cref{KnkBiglk11} tells us that $H'$ does not improve the performance for the herder over binary search when playing against the cat with strategy $C$. Assume $\lambda_{k-1}\geq 3$.

    \par If $\lambda_2=\lambda_k$, then \cref{KnkBiglk1Bigl2Islk} tells us that $H'$ does not improve the performance of the herder over binary search when playing against the cat with strategy $C$. Since $\lambda_2\leq \lambda_k$, we may conclude that $\lambda_2<\lambda_k$.

    \par If $\lambda_2=1$, then \cref{KnkBiglk1Bigl21} tells us that $H'$ does not improve the performance of the herder over binary search when playing against the cat with strategy $C$. Assume without loss of generality that $1<\lambda_2<\lambda_k$.

    \par Finally we have enough conditions to apply \cref{KnkBiglk1Bigl2Mid} to get that $H'$ does not improve the performance of the herder over binary search when playing against the cat with strategy $C$.

    \par As there is no possible $H'$ that can improve performance against $C$, we must in fact have made a faulty assumption that $c_n\neq \cat(K_n)$ for some $n\geq 2$.
\end{proof}

\begin{corollary}
    If $n\geq 2$, then strategies $C,H$ are optimal on $K_n$.
\end{corollary}
\begin{proof}
    We have now shown that no $H'$ may perform better against $C$ than $H$ does, and no $C'$ may perform better against $H$ than $C$ does,  so we have that $\cat(K_n,C,H)=\cat(K_n,C)=\cat(K_n,H)$. From \cref{lrScore}, this tells us that both strategies $C,H$ are optimal on $K_n$.
\end{proof}
Note that this means that cat herding on complete graphs is now completely solved --- given any complete graph, both players can proceed to play optimally against any adversary. In particular, each herder moves may be computed in sets, providing linear time computations which is amortized to constant time as subsequent cuts from the set require no computation. Similarly, cat moves take linear time to identify all bridges, then compute all $2$-edge-connected components, and then to identify degree/size properties. Thus all moves are computable in polynomial (and in particular linear) time. Further, we have an logarithmic time recurrence to compute the $\cat$ number of the graph which will be on the order of $n^2/3$.

%% file: Sections/openQs.tex
\section{Conclusion}
\input{Sections/Complete/discussion}

We will tease that we have partial results for cat herding on trees, and then leave the reader with the following open questions. 

Is the binary search strategy used in complete graphs optimal in the case of recursively defined graphs, such as hypercubes? An obvious herder strategy is to cut $Q_n$ into two $Q_{n-1}$s, but there exist other (smaller) cuts that break some $Q_2$ in every dimension, rapidly limiting assumptions about 2-edge-connectedness. An alternate question, for any optimal play on graph $G$, if one takes the graph induced by all the deleted edges at the moment of first disconnection, will it always be bipartite, as the binary search strategy might suggest? 

Supposing the cat restricts themself to $H$, a subgraph of $G$, when can they guarantee that they are captured on specific vertices of $H$? A proper answer to this question may allow the cat to restrict and then un-restrict themself to $H$ in $G$.

How can analysis of the game of cat herding be used to assist analysis of the combinatorial game Nowhere To Go? This is an obvious next question since this game came models an endgame of Nowhere To Go.

A slew of questions emerge on variants of the game. For instance, if we restrict the herder to only cutting edges that the cat used in the previous turn, how does the $\cat$ parameter change? Alternately, suppose we always delete all edges the cat takes for a one-player game --- how does the corresponding graph parameter change then? Another alternate game would be to require the herder to delete a vertex on each turn. It is not hard to check that for trees, this is the same game, and for cycles, the values match for $n\neq 2^k+2$, but the values vary widely for complete graphs, with a score of $n-1$ in the vertex deleting variant. One final alternate game would be to generalize the cat to always move on a path of at least $k$ edges, in which case our game is the $k=1$ variant.

%% file: Sections/Complete/discussion.tex
The herder's strategy in many of the graphs we have looked at has been in essence to `cut the graph in half'. The most precise variant of this is for the herder to choose a vertex partition $V(G)=S\cup\bar{S},\emptyset\subsetneq S\subsetneq V(G)$ minimizing the quantity $|[S,\bar{S}]|+\max(\cat(G_S),\cat(G_{\bar{S}}))$, where $G_S$ denotes the subgraph induced by $S\subsetneq V(G)$. A less precise approach still providing an upper bound would allow the herder to cut the graph in half, only looking at the number of vertices. Bisecting a graph into vertex sets $S\cup \bar{S}$ where $|S|=\lfloor |V(G)|/2\rfloor$ while minimizing $|[S,\bar{S}]|$ is known to be NP-complete\cite{garey1979computers}, and the size of $[S,\bar{S}]$ is known as bisection width. The bisection width has had multiple approximation algorithms\cite{BERRY199927,715914,goldberg1984minimal,6771089,367082} and has been generalized to hypergraphs\cite{10.1145/3529090}. Bisection width has also been given spectral bounds\cite{10.1147/rd.175.0420} and shown to typically grow linearly with vertices on cubic graphs\cite{clark_entringer_1989} and sufficiently dense random graphs\cite{randomBisect}. We will note that while vertex bisection can provide an upper bound, bisection is not generally a good approach for cat herding; consider stars $S_n$ have bisection width asymptotically $n/2$ and cat number always $2$.

%% file: Appendix/appendixDiscuss.tex
Included for completeness is the code used to verify the small cases of \Cref{KnSmallCases}. The code is written in a Jupyter notebook running a Sage Worksheet server. It is composed of three main blocks. First, we generate all of the requisite partitions. Second, we construct an automated checker for the partitions and run it against them all.

%% file: Appendix/smallCasesComplete.tex
    \begin{tcolorbox}[breakable, size=fbox, boxrule=1pt, pad at break*=1mm,colback=cellbackground, colframe=cellborder]
\prompt{In}{incolor}{1}{\boxspacing}
\begin{Verbatim}[commandchars=\\\{\}]
\PY{n}{partitions} \PY{o}{=} \PY{p}{[}\PY{p}{]}\PY{p}{;}
\PY{c+c1}{\PYZsh{}1. n\PYZlt{}9:}
\PY{k}{for} \PY{n}{n} \PY{o+ow}{in} \PY{p}{[}\PY{l+m+mf}{2.}\PY{l+m+mf}{.8}\PY{p}{]}\PY{p}{:}
    \PY{k}{for} \PY{n}{p} \PY{o+ow}{in} \PY{n}{Partitions}\PY{p}{(}\PY{n}{n}\PY{p}{)}\PY{p}{:}
        \PY{k}{if} \PY{l+m+mi}{2} \PY{o+ow}{in} \PY{n}{p} \PY{o+ow}{or} \PY{n+nb}{len}\PY{p}{(}\PY{n}{p}\PY{p}{)}\PY{o}{==}\PY{l+m+mi}{1}\PY{p}{:}
            \PY{k}{continue}\PY{p}{;}
        \PY{n}{partitions}\PY{o}{.}\PY{n}{append}\PY{p}{(}\PY{n}{p}\PY{p}{)}\PY{p}{;}
\PY{c+c1}{\PYZsh{}2. \PYZbs{}lambda\PYZus{}\PYZob{}k\PYZhy{}1\PYZcb{}=1, k\PYZbs{}in [3,4,5,6,7,8,9,10,11,12], \PYZbs{}lambda\PYZus{}k\PYZbs{}leq [56,14,9,6,5,4,3,1,1,1]}
\PY{k}{for} \PY{n}{k} \PY{o+ow}{in} \PY{p}{[}\PY{l+m+mi}{3}\PY{p}{,}\PY{l+m+mi}{4}\PY{p}{,}\PY{l+m+mi}{5}\PY{p}{,}\PY{l+m+mi}{6}\PY{p}{,}\PY{l+m+mi}{7}\PY{p}{,}\PY{l+m+mi}{8}\PY{p}{,}\PY{l+m+mi}{9}\PY{p}{,}\PY{l+m+mi}{10}\PY{p}{,}\PY{l+m+mi}{11}\PY{p}{,}\PY{l+m+mi}{12}\PY{p}{]}\PY{p}{:}
    \PY{n}{lkMax} \PY{o}{=} \PY{p}{\PYZob{}}\PY{l+m+mi}{3}\PY{p}{:} \PY{l+m+mi}{56}\PY{p}{,} \PY{l+m+mi}{4}\PY{p}{:} \PY{l+m+mi}{14}\PY{p}{,} \PY{l+m+mi}{5}\PY{p}{:} \PY{l+m+mi}{9}\PY{p}{,} \PY{l+m+mi}{6}\PY{p}{:} \PY{l+m+mi}{6}\PY{p}{,} \PY{l+m+mi}{7}\PY{p}{:} \PY{l+m+mi}{5}\PY{p}{,} \PY{l+m+mi}{8}\PY{p}{:} \PY{l+m+mi}{4}\PY{p}{,} \PY{l+m+mi}{9}\PY{p}{:} \PY{l+m+mi}{3}\PY{p}{,} \PY{l+m+mi}{10}\PY{p}{:} \PY{l+m+mi}{1}\PY{p}{,} \PY{l+m+mi}{11}\PY{p}{:} \PY{l+m+mi}{1}\PY{p}{,} \PY{l+m+mi}{12}\PY{p}{:} \PY{l+m+mi}{1}\PY{p}{\PYZcb{}}\PY{p}{[}\PY{n}{k}\PY{p}{]}\PY{p}{;}
    \PY{k}{for} \PY{n}{lk} \PY{o+ow}{in} \PY{p}{[}\PY{l+m+mf}{1.}\PY{o}{.}\PY{n}{lkMax}\PY{p}{]}\PY{p}{:}
        \PY{k}{if} \PY{n}{lk}\PY{o}{==}\PY{l+m+mi}{2}\PY{p}{:}
            \PY{k}{continue}\PY{p}{;}
        \PY{n}{p} \PY{o}{=} \PY{p}{[}\PY{n}{lk}\PY{p}{]}\PY{p}{;}
        \PY{k}{while} \PY{n+nb}{len}\PY{p}{(}\PY{n}{p}\PY{p}{)}\PY{o}{\PYZlt{}}\PY{n}{k}\PY{p}{:}
            \PY{n}{p}\PY{o}{.}\PY{n}{append}\PY{p}{(}\PY{l+m+mi}{1}\PY{p}{)}\PY{p}{;}
        \PY{n}{partitions}\PY{o}{.}\PY{n}{append}\PY{p}{(}\PY{n}{p}\PY{p}{)}\PY{p}{;}
\PY{c+c1}{\PYZsh{}3. \PYZbs{}lambda\PYZus{}3=\PYZbs{}lambda\PYZus{}k,\PYZbs{}lambda\PYZus{}2=1, k\PYZbs{}in [3,4], \PYZbs{}lambda\PYZus{}3\PYZlt{}[57,35]}
\PY{k}{for} \PY{n}{l3} \PY{o+ow}{in} \PY{p}{[}\PY{l+m+mf}{1.}\PY{l+m+mf}{.56}\PY{p}{]}\PY{p}{:}
    \PY{k}{if} \PY{n}{l3}\PY{o}{==}\PY{l+m+mi}{2}\PY{p}{:}
        \PY{k}{continue}\PY{p}{;}
    \PY{n}{p} \PY{o}{=} \PY{p}{[}\PY{n}{l3}\PY{p}{,}\PY{l+m+mi}{1}\PY{p}{,}\PY{l+m+mi}{1}\PY{p}{]}\PY{p}{;}
    \PY{n}{partitions}\PY{o}{.}\PY{n}{append}\PY{p}{(}\PY{n}{p}\PY{p}{)}\PY{p}{;}
\PY{k}{for} \PY{n}{l4} \PY{o+ow}{in} \PY{p}{[}\PY{l+m+mf}{1.}\PY{l+m+mf}{.34}\PY{p}{]}\PY{p}{:}
    \PY{k}{if} \PY{n}{l4}\PY{o}{==}\PY{l+m+mi}{2}\PY{p}{:}
        \PY{k}{continue}\PY{p}{;}
    \PY{n}{p} \PY{o}{=} \PY{p}{[}\PY{n}{l4}\PY{p}{,}\PY{n}{l4}\PY{p}{,}\PY{l+m+mi}{1}\PY{p}{,}\PY{l+m+mi}{1}\PY{p}{]}\PY{p}{;}
    \PY{n}{partitions}\PY{o}{.}\PY{n}{append}\PY{p}{(}\PY{n}{p}\PY{p}{)}\PY{p}{;}
\PY{c+c1}{\PYZsh{}k=2, n\PYZlt{}27}
\PY{k}{for} \PY{n}{n} \PY{o+ow}{in} \PY{p}{[}\PY{l+m+mf}{9.}\PY{l+m+mf}{.27}\PY{p}{]}\PY{p}{:}
    \PY{k}{for} \PY{n}{l1} \PY{o+ow}{in} \PY{p}{[}\PY{l+m+mf}{1.}\PY{o}{.}\PY{n}{floor}\PY{p}{(}\PY{n}{n}\PY{o}{/}\PY{l+m+mi}{2}\PY{p}{)}\PY{p}{]}\PY{p}{:}
        \PY{k}{if} \PY{n}{l1}\PY{o}{==}\PY{l+m+mi}{2}\PY{p}{:}
            \PY{k}{continue}\PY{p}{;}
        \PY{n}{p} \PY{o}{=} \PY{p}{[}\PY{n}{n}\PY{o}{\PYZhy{}}\PY{n}{l1}\PY{p}{,}\PY{n}{l1}\PY{p}{]}
        \PY{n}{partitions}\PY{o}{.}\PY{n}{append}\PY{p}{(}\PY{n}{p}\PY{p}{)}
\PY{c+c1}{\PYZsh{}k=3, \PYZbs{}lambda\PYZus{}k=\PYZbs{}lambda\PYZus{}2, \PYZbs{}lambda\PYZus{}1\PYZbs{}in [3,4,5,6], \PYZbs{}lambda\PYZus{}2\PYZbs{}leq [8,6,6,6]}
\PY{k}{for} \PY{n}{l1} \PY{o+ow}{in} \PY{p}{[}\PY{l+m+mi}{3}\PY{p}{,}\PY{l+m+mi}{4}\PY{p}{,}\PY{l+m+mi}{5}\PY{p}{,}\PY{l+m+mi}{6}\PY{p}{]}\PY{p}{:}
    \PY{n}{l2Max} \PY{o}{=} \PY{p}{\PYZob{}}\PY{l+m+mi}{3}\PY{p}{:} \PY{l+m+mi}{8}\PY{p}{,} \PY{l+m+mi}{4}\PY{p}{:} \PY{l+m+mi}{6}\PY{p}{,} \PY{l+m+mi}{5}\PY{p}{:} \PY{l+m+mi}{6}\PY{p}{,} \PY{l+m+mi}{6}\PY{p}{:} \PY{l+m+mi}{6}\PY{p}{\PYZcb{}}\PY{p}{[}\PY{n}{l1}\PY{p}{]}\PY{p}{;}
    \PY{k}{for} \PY{n}{l2} \PY{o+ow}{in} \PY{p}{[}\PY{l+m+mf}{1.}\PY{o}{.}\PY{n}{l2Max}\PY{p}{]}\PY{p}{:}
        \PY{k}{if} \PY{n}{l2}\PY{o}{==}\PY{l+m+mi}{2} \PY{o+ow}{or} \PY{n}{l1}\PY{o}{\PYZgt{}}\PY{n}{l2}\PY{p}{:}
            \PY{k}{continue}\PY{p}{;}
        \PY{n}{p} \PY{o}{=} \PY{p}{[}\PY{n}{l2}\PY{p}{,}\PY{n}{l2}\PY{p}{,}\PY{n}{l1}\PY{p}{]}\PY{p}{;}
        \PY{n}{partitions}\PY{o}{.}\PY{n}{append}\PY{p}{(}\PY{n}{p}\PY{p}{)}\PY{p}{;}
\PY{c+c1}{\PYZsh{}k=4, \PYZbs{}lambda\PYZus{}k=\PYZbs{}lambda\PYZus{}2, \PYZbs{}lambda\PYZus{}2\PYZbs{}leq 4.}
\PY{k}{for} \PY{n}{l2} \PY{o+ow}{in} \PY{p}{[}\PY{l+m+mi}{1}\PY{p}{,}\PY{l+m+mi}{3}\PY{p}{,}\PY{l+m+mi}{4}\PY{p}{]}\PY{p}{:}
    \PY{k}{for} \PY{n}{l1} \PY{o+ow}{in} \PY{p}{[}\PY{l+m+mf}{1.}\PY{o}{.}\PY{n}{l2}\PY{p}{]}\PY{p}{:}
        \PY{k}{if} \PY{n}{l1}\PY{o}{==}\PY{l+m+mi}{2}\PY{p}{:}
            \PY{k}{continue}\PY{p}{;}
        \PY{n}{p} \PY{o}{=} \PY{p}{[}\PY{n}{l2}\PY{p}{,}\PY{n}{l2}\PY{p}{,}\PY{n}{l2}\PY{p}{,}\PY{n}{l1}\PY{p}{]}\PY{p}{;}
        \PY{n}{partitions}\PY{o}{.}\PY{n}{append}\PY{p}{(}\PY{n}{p}\PY{p}{)}\PY{p}{;}
\PY{c+c1}{\PYZsh{}k=5, \PYZbs{}lambda\PYZus{}k=\PYZbs{}lambda\PYZus{}2, \PYZbs{}lambda\PYZus{}1=1,\PYZbs{}lambda\PYZus{}2=3}
\PY{n}{partitions}\PY{o}{.}\PY{n}{append}\PY{p}{(}\PY{p}{[}\PY{l+m+mi}{3}\PY{p}{,}\PY{l+m+mi}{3}\PY{p}{,}\PY{l+m+mi}{3}\PY{p}{,}\PY{l+m+mi}{3}\PY{p}{,}\PY{l+m+mi}{1}\PY{p}{]}\PY{p}{)}\PY{p}{;}

\PY{n+nb}{print}\PY{p}{(}\PY{l+s+s2}{\PYZdq{}}\PY{l+s+s2}{Generated }\PY{l+s+s2}{\PYZdq{}} \PY{o}{+} \PY{n+nb}{str}\PY{p}{(}\PY{n+nb}{len}\PY{p}{(}\PY{n}{partitions}\PY{p}{)}\PY{p}{)} \PY{o}{+} \PY{l+s+s2}{\PYZdq{}}\PY{l+s+s2}{ partitions.}\PY{l+s+s2}{\PYZdq{}}\PY{p}{)}\PY{p}{;}
\end{Verbatim}
\end{tcolorbox}

    \begin{Verbatim}[commandchars=\\\{\}]
Generated 376 partitions.
    \end{Verbatim}

    \begin{tcolorbox}[breakable, size=fbox, boxrule=1pt, pad at break*=1mm,colback=cellbackground, colframe=cellborder]
\prompt{In}{incolor}{2}{\boxspacing}
\begin{Verbatim}[commandchars=\\\{\}]
\PY{n+nd}{@CachedFunction}
\PY{k}{def} \PY{n+nf}{c}\PY{p}{(}\PY{n}{n}\PY{p}{)}\PY{p}{:}
    \PY{k}{if} \PY{n}{n}\PY{o}{\PYZlt{}}\PY{o}{=}\PY{l+m+mi}{3}\PY{p}{:}
        \PY{k}{return} \PY{n}{n}\PY{o}{\PYZhy{}}\PY{l+m+mi}{1}\PY{p}{;}
    \PY{k}{return} \PY{p}{(}\PY{n}{floor}\PY{p}{(}\PY{n}{n}\PY{o}{/}\PY{l+m+mi}{2}\PY{p}{)}\PY{o}{*}\PY{n}{ceil}\PY{p}{(}\PY{n}{n}\PY{o}{/}\PY{l+m+mi}{2}\PY{p}{)}\PY{p}{)}\PY{o}{+}\PY{n}{c}\PY{p}{(}\PY{n}{ceil}\PY{p}{(}\PY{n}{n}\PY{o}{/}\PY{l+m+mi}{2}\PY{p}{)}\PY{p}{)}\PY{p}{;}

\PY{k}{def} \PY{n+nf}{testPartition}\PY{p}{(}\PY{n}{p}\PY{p}{)}\PY{p}{:}
    \PY{n}{score} \PY{o}{=} \PY{l+m+mi}{0}\PY{p}{;}
    \PY{c+c1}{\PYZsh{}Fully disconnect partition.}
    \PY{k}{for} \PY{n}{i} \PY{o+ow}{in} \PY{n+nb}{range}\PY{p}{(}\PY{n+nb}{len}\PY{p}{(}\PY{n}{p}\PY{p}{)}\PY{p}{)}\PY{p}{:}
        \PY{k}{for} \PY{n}{j} \PY{o+ow}{in} \PY{n+nb}{range}\PY{p}{(}\PY{n}{i}\PY{o}{+}\PY{l+m+mi}{1}\PY{p}{,}\PY{n+nb}{len}\PY{p}{(}\PY{n}{p}\PY{p}{)}\PY{p}{)}\PY{p}{:}
            \PY{n}{score} \PY{o}{=} \PY{n}{score} \PY{o}{+} \PY{n}{p}\PY{p}{[}\PY{n}{i}\PY{p}{]}\PY{o}{*}\PY{n}{p}\PY{p}{[}\PY{n}{j}\PY{p}{]}\PY{p}{;}
    \PY{c+c1}{\PYZsh{}Don\PYZsq{}t count the bridges.}
    \PY{n}{score} \PY{o}{=} \PY{n}{score} \PY{o}{\PYZhy{}} \PY{p}{(}\PY{n+nb}{len}\PY{p}{(}\PY{n}{p}\PY{p}{)}\PY{o}{\PYZhy{}}\PY{l+m+mi}{1}\PY{p}{)}\PY{p}{;}
    \PY{c+c1}{\PYZsh{}Do count the recursive cost}
    \PY{n}{score} \PY{o}{=} \PY{n}{score} \PY{o}{+} \PY{n}{c}\PY{p}{(}\PY{n}{p}\PY{p}{[}\PY{l+m+mi}{0}\PY{p}{]}\PY{p}{)}\PY{p}{;}
    \PY{n}{n} \PY{o}{=} \PY{n+nb}{sum}\PY{p}{(}\PY{n}{p}\PY{p}{)}
    \PY{c+c1}{\PYZsh{}Final adjustments:}
    \PY{k}{if} \PY{n}{p}\PY{p}{[}\PY{l+m+mi}{0}\PY{p}{]}\PY{o}{==}\PY{l+m+mi}{1}\PY{p}{:}
        \PY{c+c1}{\PYZsh{}Recursive score is 0, but cat now plays path strategy.}
        \PY{k}{if} \PY{n+nb}{sum}\PY{p}{(}\PY{n}{p}\PY{p}{)}\PY{o}{\PYZpc{}}\PY{k}{2}==0:
            \PY{n}{score} \PY{o}{=} \PY{n}{score} \PY{o}{+} \PY{n}{ceil}\PY{p}{(}\PY{n}{log}\PY{p}{(}\PY{n}{n}\PY{p}{,}\PY{l+m+mi}{2}\PY{p}{)}\PY{p}{)}\PY{p}{;} \PY{c+c1}{\PYZsh{}On even paths, cat may move to center.}
        \PY{k}{else}\PY{p}{:}
            \PY{n}{score} \PY{o}{=} \PY{n}{score} \PY{o}{+} \PY{n}{ceil}\PY{p}{(}\PY{n}{log}\PY{p}{(}\PY{n}{n}\PY{p}{,}\PY{l+m+mi}{2}\PY{p}{)}\PY{p}{)}\PY{o}{\PYZhy{}}\PY{l+m+mi}{1}\PY{p}{;} \PY{c+c1}{\PYZsh{}On odd paths, cat may not have center available, but can still score within one of optimal path score.}
    \PY{k}{elif} \PY{n}{p}\PY{o}{==}\PY{p}{[}\PY{l+m+mi}{3}\PY{p}{,}\PY{l+m+mi}{3}\PY{p}{,}\PY{l+m+mi}{1}\PY{p}{]}\PY{p}{:}
        \PY{c+c1}{\PYZsh{}Cat moves to center. On any cut, cat plays in [3,1] and scores an extra point per the next argument.}
        \PY{n}{score} \PY{o}{=} \PY{n}{score} \PY{o}{+} \PY{l+m+mi}{2}\PY{p}{;}
    \PY{k}{elif} \PY{n}{p}\PY{p}{[}\PY{l+m+mi}{0}\PY{p}{]}\PY{o}{==}\PY{l+m+mi}{3}\PY{p}{:}
        \PY{c+c1}{\PYZsh{}Recursive score is 2, but cat can move to K3 non\PYZhy{}adjacent to outgoing edge. If next cut in K3, then cat moves to vertex in K3 adjacent to outgoing edge, center of a P3. If next cut out of K3, cat plays in K3. Either way, cat scores 3 more, not 2 more as expected.}
        \PY{n}{score} \PY{o}{=} \PY{n}{score} \PY{o}{+} \PY{l+m+mi}{1}\PY{p}{;}
    \PY{k}{elif} \PY{n}{p}\PY{o}{==}\PY{p}{[}\PY{l+m+mi}{6}\PY{p}{,}\PY{l+m+mi}{1}\PY{p}{]} \PY{o+ow}{or} \PY{n}{p}\PY{o}{==}\PY{p}{[}\PY{l+m+mi}{12}\PY{p}{,}\PY{l+m+mi}{1}\PY{p}{]} \PY{o+ow}{or} \PY{n}{p}\PY{o}{==}\PY{p}{[}\PY{l+m+mi}{24}\PY{p}{,}\PY{l+m+mi}{1}\PY{p}{]}\PY{p}{:}
        \PY{c+c1}{\PYZsh{}Cost to reach is 5. K6 would naively cost 11, we must show K7 gets 17. Cat plays in K6 until next critical graph.}
        \PY{c+c1}{\PYZsh{}If herder cuts bridge, get the extra point. Else herder eventually reaches a different partition of 7 containing a 1, all of which are verified below to score at least c(7)=17 total.}
        \PY{n}{score} \PY{o}{=} \PY{n}{score} \PY{o}{+} \PY{l+m+mi}{1}\PY{p}{;}
    \PY{k}{elif} \PY{n}{p}\PY{p}{[}\PY{l+m+mi}{0}\PY{p}{]}\PY{o}{\PYZhy{}}\PY{n}{p}\PY{p}{[}\PY{l+m+mi}{1}\PY{p}{]}\PY{o}{\PYZlt{}}\PY{o}{=}\PY{l+m+mi}{1}\PY{p}{:}
        \PY{c+c1}{\PYZsh{}Cat uses strategy of staying in p[0] until they need to move to p[1] to score 1 extra, less the bridge is cut}
        \PY{n}{score} \PY{o}{=} \PY{n}{score} \PY{o}{+} \PY{l+m+mi}{1}\PY{p}{;}
    \PY{k}{if} \PY{n}{score} \PY{o}{\PYZlt{}} \PY{n}{c}\PY{p}{(}\PY{n}{n}\PY{p}{)}\PY{p}{:} \PY{c+c1}{\PYZsh{}print anything where the score lower bound doesn\PYZsq{}t at least meet c(n)}
        \PY{n+nb}{print}\PY{p}{(}\PY{n}{p}\PY{p}{,} \PY{n}{score}\PY{p}{,} \PY{n}{c}\PY{p}{(}\PY{n}{n}\PY{p}{)}\PY{p}{)}\PY{p}{;}
\PY{k}{for} \PY{n}{p} \PY{o+ow}{in} \PY{n}{partitions}\PY{p}{:}
    \PY{n}{testPartition}\PY{p}{(}\PY{n}{p}\PY{p}{)}\PY{p}{;}
\PY{n+nb}{print}\PY{p}{(}\PY{l+s+s2}{\PYZdq{}}\PY{l+s+s2}{All }\PY{l+s+s2}{\PYZdq{}} \PY{o}{+} \PY{n+nb}{str}\PY{p}{(}\PY{n+nb}{len}\PY{p}{(}\PY{n}{partitions}\PY{p}{)}\PY{p}{)} \PY{o}{+} \PY{l+s+s2}{\PYZdq{}}\PY{l+s+s2}{ partitions tested!}\PY{l+s+s2}{\PYZdq{}}\PY{p}{)}\PY{p}{;}
\end{Verbatim}
\end{tcolorbox}

    \begin{Verbatim}[commandchars=\\\{\}]
All 376 partitions tested!
    \end{Verbatim}